\documentclass[journal,twoside,web]{ieeecolor}
\usepackage{generic}
\usepackage{cite}
\usepackage{amsmath,amssymb,amsfonts}
\usepackage{algorithmic}
\usepackage{graphicx}
\usepackage{textcomp}
\def\BibTeX{{\rm B\kern-.05em{\sc i\kern-.025em b}\kern-.08em
    T\kern-.1667em\lower.7ex\hbox{E}\kern-.125emX}}
\usepackage{mathtools}
\usepackage{subcaption}
\usepackage{algorithm}
	
\usepackage{soul}

\newtheorem{theorem}{Theorem}[section]
\newtheorem{lemma}{Lemma}[section]

\newtheorem{corollary}{Corollary}[section]

\newtheorem{remark}{Remark}

\DeclareMathOperator*{\argmin}{arg\,min}

\usepackage{tikz}
\usetikzlibrary{arrows}
\usetikzlibrary{positioning}
\usetikzlibrary{decorations}
\usetikzlibrary{calc,patterns,angles,quotes,graphs}
\usetikzlibrary{decorations.pathreplacing,calligraphy}

\let\labelindent\relax
\usepackage{enumitem}

\begin{document}
\title{Polyhedral restrictions of feasibility regions in optimal power flow for distribution networks}
\author{M.H.M.\ Christianen, S.\ van Kempen, M.\ Vlasiou and B.\ Zwart
\thanks{This paper has been submitted for review on 19/10/2023. This research is supported by the Dutch Research Council through the TOP programme under contract number 613.001.801.}
\thanks{M.H.M.\ Christianen and S.\ van Kempen are with the Eindhoven University of Technology, The Netherlands (e-mail: m.h.m.christianen@tue.nl; s.f.m.v.kempen@tue.nl).}
\thanks{M.\ Vlasiou is with the University of Twente and the Eindhoven University of Technology, The Netherlands, (e-mail: m.vlasiou@utwente.nl).}
\thanks{B.\ Zwart is with the Centrum Wiskunde en Informatica, Amsterdam, The Netherlands and also with the Eindhoven University of Technology, Eindhoven, The Netherlands (e-mail: bert.zwart@cwi.nl).}}

\maketitle

\maketitle

\begin{abstract}
The optimal power flow (OPF) problem is one of the most fundamental problems in power system operations. The non-linear alternating current (AC) power flow equations that model different physical laws (together with operational constraints) lay the foundation for the feasibility region of the OPF problem.
While significant research has focused on convex relaxations, which are approaches to solve an OPF problem by enlarging the true feasibility region, the opposite approach of convex restrictions offers valuable insights as well. Convex restrictions, including polyhedral restrictions, reduce the true feasible region to a convex region, ensuring that it contains only feasible points.
In this work, we develop a sequential optimization method that offers a scalable way to obtain (bounds on) solutions to OPF problems for distribution networks.
To do so, we first develop sufficient conditions for the existence of feasible power flow solutions in the neighborhood of a specific (feasible) operating point in distribution networks, and second, based on these conditions, we construct a polyhedral restriction of the feasibility region. 
Our numerical results demonstrate the efficacy of the sequential optimization method as an alternative to existing approaches to obtain (bounds on) solutions to OPF problems for distribution networks. By construction, the optimization problems can be solved in polynomial time and are guaranteed to have feasible solutions.


\end{abstract}


\section{Introduction}\label{sec:introduction}
The optimal power flow (OPF) problem aims to determine power generations and demands (loads) in an electricity network  to meet certain objectives such as minimizing generation cost, that are in the \emph{feasibility region}. A set of loads are in the feasibility region when they satisfy two sets of constraints. First, the non-linear AC power flow equations that model physical laws, and second, operational constraints, such as capacity and security constraints on voltages and power flows. Finding solutions to the first set of constraints is one of the main hurdles in solving OPF problems, which has been one of the most fundamental tasks in system operations \cite{Bose2011a, Gan2015b}.

Recently, the OPF problem has become more important for \emph{distribution networks}, due to the emergence of distributed generation (e.g., solar panels) and controllable loads (e.g., electric vehicles (EVs)). Unexpected changes or variations in generation or consumption of electric power make it difficult to predict distributed generation, potentially leading to constraint violations. To deal with these changes, solving the OPF problem in real-time is required.

We introduce a scalable sequential optimization method for obtaining (bounds on) solutions to OPF problems for distribution networks. This method is based on two key components. First, we develop sufficient conditions for the existence of feasible power flow solutions around a specific feasible operating point in distribution networks, and subsequently, we construct a polyhedral restriction of the true feasibility region using these conditions. Such a polyhedral restriction has two desirable characteristics. First, by construction, all points within the restriction are AC power flow feasible. This is in contrast to approximations (linearizations) and convex relaxations, which in general do not result in feasible solutions to the AC power flow equations. Second, optimizing a certain objective function of an OPF problem over a polyhedral region is as computationally efficient as the alternative path of linearizing power flow equations. Last, we compare the polyhedral restriction with existing convex restrictions and check the validity of the sequential optimization method by various numerical experiments on different distribution networks.


Finding solutions to the power flow equations is challenging due to their non-linearity. Solutions may not exist, and if they do, finding them is difficult \cite{Dvijotham2017a}.

For classical power flow algorithms there are only a few theoretical guarantees for the convergence to solutions, if they exist \cite{Giraldo2022}. More recently, new algorithms based on fixed point iterations have been introduced, since various fixed point theorems can guarantee existence and/or uniqueness of solutions to the power flow equations. In \cite{Yu2015a}, an extension of the work in \cite{Bolognani2016a}, a fixed point load-flow method is presented for radial distribution networks with a single slack bus. In the same paper, sufficient conditions are given to guarantee the existence and uniqueness of the solution. These sufficient conditions are further improved in \cite{Wang2018}. Furthermore, these conditions control the power injections in order to enforce the important voltage drop constraint in distribution networks. While we are in the same setting as in the previous mentioned works, a distinguishing factor of our work lies in the conditions established. Our approach yields affine constraints in terms of the loads, setting it apart from the work done in \cite{Yu2015a,Bolognani2016a, Wang2018}.

Having conditions on the existence of feasible solutions to the power flow equations in OPF problems is valuable, but it does not solve OPF problems. Embedding the non-linear power flow equations in optimization problems yield difficulties in solving these problems due to nonconvexity. In general, there are three ways to deal with this challenge: (i) approximate (and in most cases, linearize) the power flow equations, (ii) relax the feasibility region by making it convex and (iii) restrict the feasibility region by making it convex.

The last type of OPF methods contrasts with convex relaxations, such as Semidefinite Programming (SDP) or Second-Order Cone Programming (SOCP) relaxations \cite{Kocuk2016, Gan2015a}, since the methods restrict the non-convex OPF feasible region to become convex by removing feasible points. The solution obtained through convex restrictions are always feasible with respect to the AC power flow equations and are always a bound on the solution for the original OPF problem.

The related literature on convex restrictions is different from our work in two directions. It either lacks a guarantee on feasibility of the power flow equations or, if this is guaranteed, fails to yield affine constraints in terms of the loads.

In early work \cite{Wu1982} on convex restrictions of power flow equations, a ``security'' region is constructed. This set consists of load demands and power generations for which the power flow equations and the security constraints imposed by operating constraints are satisfied. However, this security region is constructed with the use of approximations of the power flow equations. In \cite{Wei2017}, a sequential convex optimization method is proposed to solve OPF problems over radial networks. However, the convex functions used to approximate the non-convex parts are not necessarily restrictions of the feasible region and therefore do no guarantee feasibility.

In \cite{Dvijotham2015a}, a new approach for the construction of convex regions where the AC power flow equations are guaranteed to have feasible solutions is introduced. The construction is based on solving SDP problems. The resulting regions, polytopes or ellipsoids, can be efficiently used for assessment of system security in the presence of uncertainty. Similar work has been done by the authors in \cite{Lee2019}, \cite{Lee2020}, \cite{Lee2021a}. In \cite{Lee2019}, a framework of constructing convex restrictions is built and applied to the power flow feasibility problem. The convex restrictions are used for the identification of a feasible path between two points \cite{Lee2020}, and extended to a robust convex restriction that accounts for uncertainty in the power injections \cite{Lee2021a}.




To summarize, the main contributions of this paper are:
\begin{enumerate}[wide,nosep]
\item We restrict the feasibility region of the AC-OPF problem to a feasible region with only affine constraints. The formulation guarantees that there exists at least one solution to the non-linear power flow equations that also satisfies operational constraints, including the important voltage drop constraints.
\item Using the polyhedral restriction in different OPF problems, we propose a sequential optimization method which in each iteration (i) constructs a polyhedral restriction around a feasible point and (ii) solves an OPF problem to obtain a better feasible point (in view of the objective function). 
\item We show the conservativeness of the polyhedral restriction and the applicability of the sequential optimization method using numerical experiments on different test cases.
\end{enumerate}

This paper is organized as follows. Section \ref{sec:model} introduces the model for a distribution network. Section \ref{sec:main_results} derives the sufficient conditions for the existence of at least one feasible solution to the power flow equations and operational constraints which can be used to construct a polyhedral restriction of the true feasible region, and extends this polyhedral restriction to develop a sequential optimization method to solve OPF problems. Section \ref{sec:numerical_results} demonstrates the use of this method in different scenarios. Section \ref{sec:conclusion} concludes the paper.

\section{Model description}\label{sec:model}
\subsection{Notation}
We consider a distribution network modeled by a rooted tree $\mathcal{G}=(\mathcal{N},\mathcal{E})$ with $\mathcal{N} = \{0\}\cup \{1,2,\ldots,N\}$ nodes (or \emph{buses}), where the root node $0$ is the main feeder (or \emph{slack bus}) and the remaining nodes are load nodes (or \emph{PQ buses}). An edge (or \emph{distribution line}) $(j,k)\in\mathcal{E}$ connects nodes $j$ and $k$.

Let $V_0$ denote the complex voltage, $I_0$ the complex current, and $s_0$ the complex power injected of the slack bus, let $\mathbf{V} = (V_1,V_2,\ldots,V_N)^T$ denote the vector of complex voltages, $\mathbf{I} = (I_1,I_2,\ldots,I_N)^T$ the vector of complex currents, and $\mathbf{s} =(s_1,s_2,\ldots,s_N)^T$ the vector of complex powers injected into the PQ buses (a positive value in the real or the imaginary part means that power is consumed).

Moreover, for a bus $j\in\mathcal{N}\backslash\{0\}$, power can be exclusive generation, consumption, both, or neither. The complex power consumption of node $j\in\mathcal{N}$ is $s_j^c = p_j^c+\mathrm{i}q_j^c$, where $p_j^c$ and $q_j^c$ denote the \emph{active} and \emph{reactive} power consumption. Similarly, the complex power generation is given by $s_j^g = p_j^g + \mathrm{i}q_j^g$. We use the term \emph{load} for power consumption minus generation and \emph{injection} for its logical counterpart. 

For a line $(j,k)\in\mathcal{E}$, the \emph{admittance} $y_{jk}$ and \emph{impedance} $z_{jk}$ are given by $z_{jk}=r_{jk}+\mathrm{i}x_{jk} = 1/y_{jk}$,
where $r_{jk}\geq 0,x_{jk}\geq 0$ are the 
\emph{resistance} and \emph{reactance} of line $(j,k)\in\mathcal{E}$, respectively. Note that $y_{jk} = y_{kj}$ and $z_{jk}=z_{kj}$. We introduce the admittance bus matrix $\mathbf{Y}\in\mathbb{C}^{(N+1)\times(N+1)}$ as
\begin{align}
    \mathbf{Y} = \begin{pmatrix}
    Y_{00} & \mathbf{Y}_{0L} \\ \mathbf{Y}_{L0} & \mathbf{Y}_{LL}
    \end{pmatrix}, \label{eq:Y}
\end{align} with $Y_{00}\in \mathbb{C}$, $\mathbf{Y}_{0L}\in\mathbb{C}^{1\times N}$, $\mathbf{Y}_{L0}\in\mathbb{C}^{N\times 1}$,$\mathbf{Y}_{LL}\in\mathbb{C}^{N\times N}$ where all elements $Y_{jk}$ of the matrix $\mathbf{Y}$ are defined by
\begin{align*}
    \quad Y_{jk} =   \begin{cases}-y_{jk}, & \text{if}\ j\neq k, \\ \sum_{m=0,m\neq j} y_{jm}, & \text{if}\  j=k, \end{cases}
\end{align*} and $y_{jk}=0$ if $(j,k)\notin \mathcal{E}$.

We define the \emph{impedance bus matrix \textbf{Z}} as the inverse of the \emph{reduced} admittance matrix $\mathbf{Y}_{LL}$, proven to be non-singular for a broad range of distribution networks \cite{Wang2018}; i.e.,
\begin{align}
    \mathbf{Z}:=\mathbf{Y}_{LL}^{-1}.\label{eq:Z}
\end{align} Due to general disadvantages of matrix inversion, 
we provide a different formulation of the $Z$-matrix \cite{Peterson1989},
\begin{align}
Z_{jk} = \sum_{(m,l)\in\mathcal{P}(j)\cap\mathcal{P}(k)} z_{ml} \label{eq:Z_{jk}},
\end{align} where $\mathcal{P}(j)$ is the unique path from node $j$ to the main feeder. Efficient methods for the construction of the $Z$-matrix exist \cite{Liu2008}. By splitting the real and imaginary parts in $\mathbf{Z} = \mathbf{R}+\mathrm{i}\mathbf{X}$, we define the elements in the \emph{resistance bus matrix} $\mathbf{R}$ and the \emph{reactance bus matrix} $\mathbf{X}$ similar as in \eqref{eq:Z_{jk}}, except for replacing elements $z_{ml}$ by $r_{ml}$ and $x_{ml}$, respectively. 
For a complex number $\mathbf{w} = a+b\mathrm{i}$, we use $\text{Re}(\mathbf{w})$ for its real part, $\text{Im}(\mathbf{w})$ for its imaginary part, $\mathbf{w}^*$ for its complex conjugate, $|\mathbf{w}| = \sqrt{a^2+b^2}$ for its magnitude, and for a matrix $\mathbf{W}$ of complex elements, $|\mathbf{W}|$ denotes the same operation applied per element.

\subsection{Fixed point equation for the voltages}
As known \cite{Wang2018}, the powers and currents can be expressed in matrix form as
\begin{align}
    -\begin{pmatrix}
    s_0^* \\ \mathbf{s}^*
    \end{pmatrix} & =
    \begin{pmatrix}
    V_0 & \mathbf{0} \\ 0 & \text{diag}(\mathbf{V}^*)
    \end{pmatrix}
    \begin{pmatrix}
    I_0 \\ \mathbf{I}
    \end{pmatrix}, \label{eq:complex_power} \\
    \begin{pmatrix}
    I_0 \\ \mathbf{I}
    \end{pmatrix} & = \mathbf{Y}
    \begin{pmatrix}
    V_0 \\ \mathbf{V}
    \end{pmatrix}.\label{eq:ohms_law}
\end{align} 

Equating both representations of $\mathbf{I}$ in \eqref{eq:complex_power} and \eqref{eq:ohms_law}, keeping in mind \eqref{eq:Z} and solving for $\mathbf{V}$ yields
\begin{align}
    \mathbf{V} = -V_0 \mathbf{Z} \mathbf{Y}_{L0}-\mathbf{Z}\text{diag}(\mathbf{V}^*)^{-1}\mathbf{s}^*.\label{eq:fixed_point}
\end{align} 

\begin{remark}
The radial network structure implies
\begin{align}
    -V_0\mathbf{Z}\mathbf{Y}_{L0} = V_0\mathbf{1}^N,\label{eq:load_profile_no_consumption}
\end{align} which is the load profile in case there is no power consumption nor generation in the network. This can be seen by considering the definition of the admittance bus matrix $\mathbf{Y}$ in \eqref{eq:Y}: $[\mathbf{Y}_{L0}\  \mathbf{Y}_{LL}]\mathbf{1}^{N+1} = \mathbf{Y}_{L0}+\mathbf{Y}_{LL}\mathbf{1}^N = \mathbf{0}$,
or equivalently $-\mathbf{Y}_{LL}^{-1}\mathbf{Y}_{L0} = -\mathbf{Z}\mathbf{Y}_{L0} = \mathbf{1}^N$,
if the reduced admittance matrix $\mathbf{Y}_{LL}^{-1}$ exists. Multiplying both sides by $V_0$ yields \eqref{eq:load_profile_no_consumption}.
\end{remark}

By using \eqref{eq:load_profile_no_consumption}, the power flow equations in \eqref{eq:fixed_point} can be rewritten as fixed-point equations:
\begin{align}
    \mathbf{V} = G(\mathbf{V}) := V_0\mathbf{1}^N - \mathbf{Z}\text{diag}(\mathbf{V}^*)^{-1}\mathbf{s}^*.\label{eq:fixed_point_eq}
\end{align}
Given a load profile $\mathbf{s}$, we perform the corresponding iterations
\begin{align*}
\mathbf{V}^{(k+1)} = V_0\mathbf{1}^N - \mathbf{Z}\text{diag}({\mathbf{V}^*}^{(k)})^{-1}\mathbf{s}^*
\end{align*} of \eqref{eq:fixed_point_eq}, to compute the corresponding voltage $\mathbf{V}$.

\begin{remark} Whenever we talk about a solution to the power flow equations, we actually talk about the pair $(\mathbf{V},\mathbf{s})$ that satisfies the power flow equations in \eqref{eq:fixed_point_eq}.
\end{remark}

Similarly, equating both representations of $I_0$, and solving for $s_0^*$ yields
\begin{align}
s_0^* = -V_0Y_{00}V_0 - V_0\mathbf{Y}_{0L}\mathbf{V}.\label{eq:power_root_node}
\end{align} Thus, after using \eqref{eq:fixed_point_eq}, the power injection at the slack bus can be computed.

\subsection{OPF formulation}
In this paper, we consider a traditional AC-OPF-like formulation defined by the following equations and inequalities:
\begin{subequations}\label{eq:OPF-problem}
\begin{align}
&\textrm{Power flow equations in}  \ \eqref{eq:fixed_point_eq},\label{eq:OPF-con-pfe}\\
& \underline{V_j}\leq |V_j|\leq \overline{V_j},\quad j\in\mathcal{N},\label{eq:OPF-volt-con} \\
& s_j\in\mathcal{S}_j,\quad j\in\mathcal{N}. \label{eq:OPF-load-con}
\end{align}
\end{subequations}


In addition to the power flow equations in \eqref{eq:OPF-con-pfe}, inequalities \eqref{eq:OPF-volt-con} provide the voltage drop constraints. The inequalities in \eqref{eq:OPF-volt-con} imply that the magnitude of any voltage $V_j$ falls between the lower bound $\underline{V_j}$ and the upper bound $\overline{V_j}$.  Last, Equation \eqref{eq:OPF-load-con} restricts the complex powers to some admissible set. 

More constraints might be enforced in \eqref{eq:OPF-problem} regarding other physical limitations. We refer to \cite{Stott2012} for a review of OPF requirements in real-life grids.

In what follows, we restrict the feasible region of 
\eqref{eq:OPF-problem} to another set of equations and inequalities for which we can guarantee a solution to the power flow equations in \eqref{eq:OPF-con-pfe}. To do so, we introduce a subset of all voltages $V_j$ that satisfy the inequalities \eqref{eq:OPF-volt-con} and reformulate the equations \eqref{eq:OPF-load-con}.

Inequalities \eqref{eq:OPF-volt-con} assure that the voltage magnitudes are within a certain range. We allow a $\Delta\%$ voltage deviation from the nominal voltage $\hat{V}_0$, denoted by pre-specified bounds $(1-\Delta)|\hat{V}_0|$ and $(1+\Delta)|\hat{V}_0|$. The choice of $\Delta$ can adhere to safe operating regimes of appliances or law \cite{Gan2015b}. 

In the complex plane, inequalities \eqref{eq:OPF-volt-con} define a disk. This implies that, for a feasible voltage, the vector $V_j$ must have its endpoint in the disk. However, this does not necessarily mean that the voltage angle differences are close to each other. In distribution networks, it is known that the voltage angles differences are small \cite{Giraldo2022, Carvalho2015b}.

To this end, we assume that we have the knowledge of a pair $(\hat{\mathbf{V}},\hat{\mathbf{s}})$ that satisfies the power flow equations in \eqref{eq:fixed_point_eq} and that the voltage angles differences between $\mathbf{V}$ and $\hat{\mathbf{V}}$ are small. This means that, for $\Delta\in[0,1)$, we define the set of $\Delta$-stable voltage vectors as all  vectors that satisfy the constraint $|V_k-\hat{V}_j|\leq \Delta |\hat{V}_j|$ for all $k\in\mathcal{N}\backslash\{0\}$, i.e.,
\begin{align}
    D:=\{\mathbf{V}\in\mathbb{C}^N: \|\mathbf{V}-\hat{\mathbf{V}}\|_{\infty} \leq \Delta |\hat{\mathbf{V}}| \}.\label{eq:D}
\end{align}

The set of $\Delta$-stable voltage vectors $D$ can also be represented in the complex plane. 
For all voltages $V_j$ in the set $D$, it means that the endpoint of the vector $V_j$ is contained in the ball centered around the endpoint of the vector $\hat{V}_j$ with radius $\Delta|\hat{V}_j|$. 
Notice that every voltage $V_j$ that is in the set $D$ satisfies the inequalities in \eqref{eq:OPF-volt-con}, i.e., we have a subset of voltages that satisfy the inequalities in terms of length, see inequalities \eqref{eq:OPF-volt-con}, but satisfy extra constraints in terms of direction. We make this more rigorous in Lemmas \ref{lemma:ineq_voltage_magnitude} and \ref{lemma:ineq_voltage_angles}.



Having restricted the inequalities in \eqref{eq:OPF-volt-con}, we now turn our attention to the equations in \eqref{eq:OPF-load-con}. Constraints for power consumption arise from physical properties of appliances and constraints for power generation arise from renewable energy resource capacity. For example, if $s_j$ represents a solar panel with generation capacity $\overline{p}_j$ and nominal capacity $\overline{s}_j$, then $\mathcal{S}_j = \{s\in\mathbb{C}\ |\  -\overline{p}_j \leq \text{Re}(s)\leq 0, |s|\leq \overline{s}_j \}$, or if $s_j$ represents a controllable load with constant power factor $\eta$, whose real power consumption can very from $\underline{p}_j$ to $\overline{p}_j$, then $\mathcal{S}_j = \{s\in\mathbb{C}\ |\  \underline{p}_j\leq \text{Re}(s)\leq \overline{p}_j, \text{Im}(s) = (\sqrt{1-\eta^2}\text{Re}(s))/\eta\}$, \cite{Gan2015b,Low2014}.

To reformulate the equations in \eqref{eq:OPF-load-con}, we split the variable $s_j$ into consumption $s_j^c$ and generation $s_j^g$ explicitly, such that for all $j\in\mathcal{N}$,
\begin{align*}
    &s_j = s_j^c-s_j^g, \quad s_j^c = p_j^c+\mathrm{i}q_j^c, \quad s_j^g = p_j^g+\mathrm{i}q_j^g, \\ &p_j^c,p_j^g,q_j^c,q_j^g\geq 0.
\end{align*} Then, using the subset of inequalities in \eqref{eq:OPF-volt-con} and splitting of the variables $s_j$ in consumption and generation, we can formulate the restriction of \eqref{eq:OPF-problem} as
\begin{subequations}\label{eq:OPF-adjusted-problem}
\begin{align}
&\textrm{Power flow equations in}  \ \eqref{eq:fixed_point_eq},\label{eq:OPF-con-pfe-new}\\
& |V_j-\hat{V}_j|\leq \Delta|\hat{V}_j|,\label{eq:OPF-volt-con-new} \\
& p_j^c\in\mathcal{P}^c_j, q_j^c\in\mathcal{Q}_j^c, p_j^g\in\mathcal{P}_j^g, q_j^g\in\mathcal{Q}_j^g \quad j\in\mathcal{N}, \label{eq:OPF-load-con-new}\\
& p_j^c,p_j^g,q_j^c,q_j^g\geq 0,\quad j\in\mathcal{N}. \label{eq:OPF-load-positive-new}
\end{align}
\end{subequations}

We define the feasibility region of \eqref{eq:OPF-adjusted-problem} as
\begin{multline}
    S:=\{\tilde{\mathbf{s}}\in \mathbb{R}_+^{4N}: \exists \mathbf{V}\in D\ \text{that satisfies}\  G(\mathbf{V}) = \mathbf{V}\ \text{and has}\ \\ s_j\in\mathcal{S}_j\ \text{for all}\ j\in\mathcal{N}\},\label{eq:feasibility_region_S}
\end{multline} where $\tilde{\mathbf{s}} = (\mathbf{p}^c,\mathbf{q}^c,\mathbf{p}^g,\mathbf{q}^g)^T$. This is a natural definition: the existence of a voltage $\mathbf{V}\in D$ implies the inequalities in \eqref{eq:OPF-volt-con-new}, the equality $G(\mathbf{V}) = \mathbf{V}$ ensures the satisfaction of \eqref{eq:OPF-con-pfe-new}, while the equations $s_j\in\mathcal{S}_j,\ j\in\mathcal{N}$ are reformulated in \eqref{eq:OPF-load-con-new}--\eqref{eq:OPF-load-positive-new}.

Before we state the main results in Section \ref{sec:main_results}, we give an overview of the setting. We assume knowledge of a load vector $\hat{\mathbf{s}}$ and its corresponding voltage $\hat{\mathbf{V}}$ in the feasibility region $S$. 
The conditions of the polyhedral restriction we develop are formulated in terms of $(\hat{\mathbf{V}}, \hat{\mathbf{s}})$ and a given $\mathbf{s}$ such that $s_j\in\mathcal{S}_j$ for all $j\in\mathcal{N}$, and are used to guarantee the existence of at least one solution $\mathbf{V}$ to \eqref{eq:OPF-con-pfe-new} which is ``close'' to $\hat{\mathbf{V}}$ (cf. the set of $\Delta$-stable voltage vectors in \eqref{eq:D}). This setting is especially relevant in situations where the operational constraints \eqref{eq:OPF-load-con-new}--\eqref{eq:OPF-load-positive-new}, typically redundant, give way to the more stringent constraints imposed by the power flow equations \eqref{eq:OPF-con-pfe-new}.

\section{Main results}\label{sec:main_results}
In this section, we present two main results. The first main result is the construction of a polyhedral restriction of the feasibility region $S$, for which the power flow equations in \eqref{eq:OPF-con-pfe-new} have at least one solution which satisfies all operational constraints \eqref{eq:OPF-volt-con-new}--\eqref{eq:OPF-load-positive-new}. Mathematically, the polyhedral restriction represents sufficient conditions for the existence of at least one feasible solution that can be quickly checked or even more importantly, enforced in optimization problems. Consequently, the second main result leverages these polyhedral restrictions to develop a sequential optimization approach to obtain (bounds on) solutions to OPF problems.


The first main result is a polyhedral restriction of \eqref{eq:feasibility_region_S}.

\begin{theorem}\label{thm:polyhedral_restriction}
Denote $\hat{V}_{\text{min}} := \min_j |\hat{V}_j|$. The set
\begin{multline}
    P:=\bigg\{\tilde{\mathbf{s}}\in\mathbb{R}_+^{4N}:\quad (\mathbf{A}+\Delta\mathbf{B})\tilde{\mathbf{s}} \leq
\Delta(1-\Delta)^2\hat{V}_{\text{min}}^3\mathbf{1}^{4N} + \\ + (\mathbf{A}-\Delta(\mathbf{B}+(1-\Delta)\mathbf{C}))\tilde{\hat{\mathbf{s}}}\bigg\}\label{eq:restriction_P}
\end{multline} is a polyhedral restriction of $S$, where $\mathbf{A},\mathbf{B},\mathbf{C}\in\mathbb{R}^{4N\times 4N}$ are given by
\begin{align}
\mathbf{A} = \begin{pmatrix}
-(\mathbf{R}+\mathbf{X}) & -(-\mathbf{R}+\mathbf{X}) & \mathbf{R}+\mathbf{X} & -\mathbf{R}+\mathbf{X} \\
-(-\mathbf{R}+\mathbf{X}) & -(-\mathbf{R}-\mathbf{X}) & -\mathbf{R}+\mathbf{X} & -\mathbf{R}-\mathbf{X} \\
-(\mathbf{R}-\mathbf{X}) & -(\mathbf{R}+\mathbf{X}) & \mathbf{R}-\mathbf{X} & \mathbf{R}+\mathbf{X} \\
-(-\mathbf{R}-\mathbf{X}) & -(\mathbf{R}-\mathbf{X}) & -\mathbf{R}-\mathbf{X} & \mathbf{R}-\mathbf{X}
\end{pmatrix},\label{eq:A}
\end{align}
\begin{align}
    \mathbf{B} = \mathbf{J}_{4}\otimes (\mathbf{R}+\mathbf{X}),~ \text{and}~ \mathbf{C} = \mathbf{J}_{4}\otimes |\mathbf{Z}|,\label{eq:C}
\end{align} where $\mathbf{J}_{4}$ is a $(4\times 4)$- all-ones matrix and $\otimes$ denotes the Kronecker-product.
\end{theorem}

The proof of Theorem \ref{thm:polyhedral_restriction} is structured along the same lines as the proof of \cite[Theorem 1]{Wang2018}, and makes use of the following lemmas. First, in Lemma \ref{lemma:self-mapping}, we show that for $\tilde{\mathbf{s}}\in P$, the operator $G$ is a self-map on the metric space $D$. To show that $G$ is a self-map, we use inequalities derived in Lemmas \ref{lemma:ineq_voltage_magnitude} and \ref{lemma:ineq_voltage_angles}. Last, in Lemma \ref{lemma:continuous}, we show that $G$ is continuous on $D$. We then apply Brouwer’s fixed-point theorem to conclude that for $\tilde{\mathbf{s}}\in P$, there exists a voltage vector $\mathbf{V}$ that is $\Delta$-stable and satisfies the power flow equations \eqref{eq:OPF-con-pfe-new}. We now formally state these lemmas.

\begin{lemma}\label{lemma:self-mapping}
Let $\tilde{\mathbf{s}}\in P$ and $\mathbf{V}\in D$, then $G$ is a self-map on $D$.
\end{lemma}

\begin{lemma}\label{lemma:ineq_voltage_magnitude}
Let $\mathbf{V}\in D$. Then, the magnitude of all voltages at all nodes $j\in\mathcal{N}$ are bounded as follows:
\begin{align*}
        (1-\Delta)|\hat{V}_j|\leq |V_j| \leq (1+\Delta)|\hat{V}_j|.
\end{align*}
\end{lemma}

\begin{lemma}\label{lemma:ineq_voltage_angles}
Let $\mathbf{V}\in D$. For all $j\in\mathcal{N}$, denote the voltages in their exponential form as $V_j = |V_j|\exp(\mathrm{i}\theta_j)$ and $\hat{V}_j = |\hat{V}_j|\exp(\mathrm{i}\hat{\theta}_j)$. Then, the voltage angles at all nodes $j\in\mathcal{N}$ are bounded as follows:
\begin{align}
    (1-\Delta)|\hat{V}_j| \leq &|V_j|\cos(\theta_j-\hat{\theta}_j) \leq (1+\Delta)|\hat{V}_j|, \label{eq:cos_bound} \\
    -\Delta|\hat{V}_j| \leq &|V_j|\sin(\theta_j-\hat{\theta}_j) \leq \Delta|\hat{V}_j|. \label{eq:sin_bound}
\end{align}
\end{lemma}

\begin{lemma}\label{lemma:continuous}
Let $\tilde{\mathbf{s}}\in P$ and $\mathbf{V}\in D$, then $G$ is a continuous operator on $D$.
\end{lemma}

The proofs of Lemmas \ref{lemma:self-mapping}--\ref{lemma:continuous} are given in the appendix. 

\begin{proof}[Proof of Theorem \ref{thm:polyhedral_restriction}]
Let $\tilde{\mathbf{s}}\in P$ and $\mathbf{V}\in D$. It is well-known that $(\mathbb{C}^N,\ell^{\infty})$ is a Banach space. The disk $D$ is a compact subset of $\mathbb{C}^N$. Furthermore, from Lemmas \ref{lemma:self-mapping} and \ref{lemma:continuous}, we have that $G$ is a continuous self-map on $D$. Then, from Brouwer's fixed-point theorem, it follows that $G$ has a fixed-point $\mathbf{V}$, i.e., there exists a vector $\mathbf{V}\in D$ such that $G(\mathbf{V}) = \mathbf{V}$. This means that for all complex powers $\tilde{\mathbf{s}}\in P$ that can be represented by linear terms of consumption and generation of power, there exists a voltage vector $\mathbf{V}$ that is $\Delta$-stable and satisfies the power flow equations in \eqref{eq:OPF-con-pfe-new}.
\end{proof}

We allow that $(\hat{\mathbf{V}}, \hat{\mathbf{s}})$ includes the pair $(V_0\mathbf{1}^N, \mathbf{0})$, where there is no consumption nor injection in the grid. 
In this case, 
the result of Theorem \ref{thm:polyhedral_restriction} reduces to the following corollary.

\begin{corollary} \label{cor:polyhedral_restriction} Observe that if $(\hat{\mathbf{V}},\hat{\mathbf{s}}) = (V_0\mathbf{1}^N,\mathbf{0})$, then from Theorem \ref{thm:polyhedral_restriction}, we have that $P$ reduces to
\begin{align*}
    \bigg\{\tilde{\mathbf{s}}\in\mathbb{R}_+^{4N}:\quad (\mathbf{A}+\Delta\mathbf{B})\tilde{\mathbf{s}} \leq
\Delta(1-\Delta)^2|V_0|^3\mathbf{1}^{4N}\bigg\}.
\end{align*} 
\end{corollary}

We have derived sufficient polyhedral conditions for AC power flow feasibility. Now, we provide the second main result: a sequential optimization method that solves OPF problems. The method is described in Algorithm \ref{alg:restrictions}.


Given a current feasible point $(\hat{\mathbf{V}}^{(k)},\hat{\mathbf{s}}^{(k)})$, each iteration of the algorithm (i) constructs the polyhedral restriction around the feasible point, (ii) solves an OPF problem subject to the current polyhedral constraints, and (iii) updates the new optimal (and feasible) points $(\hat{\mathbf{V}}^{(k+1)},\hat{\mathbf{s}}^{(k+1)})$, and the voltage drop parameter $\Delta$ to ensure feasibility, accordingly. The output of the algorithm is a sequence of variables $(\hat{\mathbf{V}}^{(k)},\hat{\mathbf{s}}^{(k)}), k=0,\ldots,K,$ that optimize the chosen objective function and satisfy the power flow equations and operational constraints.

\begin{algorithm}
\caption{Sequential optimization method for solving OPF with polyhedral restrictions.}
\begin{algorithmic}[1]
\STATE Initialize: Set $\mathbf{s}^{(0)}$ and $\mathbf{V}^{(0)}$ as initial feasible point, and $\Delta^{(0)}$ as desired voltage drop parameter.
\WHILE{$\|(f(\mathbf{s}^{(k+1)})-f(\mathbf{s}^{(k)}))/f(\mathbf{s}^{(k)})\|_2>\epsilon$}{}
	\STATE Construct polyhedral restriction using \eqref{eq:restriction_P}--\eqref{eq:C}.
    \STATE Solve \begin{align}
    \mathbf{s}^{(k+1)} = \argmin_{s\in P} f(\text{Re}(s)).\label{eq:different_objectives}
    \end{align}
    \STATE Solve power flow given $\mathbf{s}^{(k+1)}$ to obtain $\mathbf{V}^{(k+1)}$.
    \STATE Update $\Delta^{(k+1)} = \Delta^{(0)}-\|\mathbf{V}^{(k)}-\mathbf{V}^{(0)}\|_{\infty}$.
	\STATE $k:=k+1$.
\ENDWHILE
\STATE Return $\mathbf{s}^{(1)},\ldots,\mathbf{s}^{(K)}$.
\end{algorithmic}
\label{alg:restrictions}
\end{algorithm}

Depending on the problem setting, we might want to consider a variety of functions $f$ in \eqref{eq:different_objectives}. We provide an example in classical optimal power flow: given the current state of the network $(\hat{\mathbf{V}},\hat{\mathbf{s}})$, we want to find a lower-cost operating point $(\mathbf{V},\mathbf{s})$ while satisfying the power flow equations and operational constraints. In the formulation in \eqref{eq:different_objectives}, the OPF problem is solved by setting the objective to minimize the generation of active load in the system. 
The algorithm solves the OPF using the polyhedral restriction from Theorem \ref{thm:polyhedral_restriction} and iterates by setting the solution to the new feasible point.


\section{Numerical experiments}\label{sec:numerical_results}
This section conducts numerical experiments on different networks outlined in Section \ref{subsec:models}. It compares the polyhedral restriction and another approach that develops sufficient conditions for the existence of feasible power flow solutions in Section \ref{subsec:conservativeness}, and it shows the efficacy of Algorithm \ref{alg:restrictions} on the different test cases in Section \ref{subsec:poly_restriction}. 

\subsection{Distribution network models}\label{subsec:models}
We validate our algorithm with three distribution network test cases: the first two are small models (one and two load buses) demonstrating theoretical concepts, and the third models a real-life distribution network.

\paragraph{Two-node network} We consider a line network with $N=2$ nodes and line $\mathcal{E}=\{(0,1)\}$, with impedance $z_{01}$. Here, node $1$ has power consumption $s_1^c$ and generation $s_1^g$. 

\paragraph{Three-node network} We examine a line network with $N=3$ nodes and lines $\mathcal{E}=\{(0,1),(1,2)\}$, with impedances $z_{01}$ and $z_{12}$. The power consumption at node $j$ is $s_j^c$, the power generation is $s_j^g$ and the net load is $s_j=s_j^c-s_j^g$ for $j=1,2$. 

%
Here, we consider two different optimization problems: (i) maximize the active power consumption and (ii) minimize the active power generation in the network. For (i), the optimal solution coincides with the solutions provided by the SOCP relaxation and the polyhedral restriction. While, for (ii), the SOCP relaxation gives an infeasible solution and the polyhedral restriction provides a feasible solution, although not optimal. The optimal solution can, for such a small network, be computed using GloptiPoly \cite{Henrion2002}.

\paragraph{SCE-47 network}
This network is a model of a real-life distribution grid obtained from the Californian electricity supply company Southern California Edison. For details about the SCE-47 network, we refer to \cite{Gan2015b}.

For this network, we consider similar optimization problems as in the three-node network. Here, we see that all solutions obtained using the SOCP relaxation are infeasible, while the polyhedral restriction provides feasible solutions and bounds on the true optimal solutions.


\subsection{Comparison of different feasibility regions}\label{subsec:conservativeness}
In this section, we compare the feasibility regions as they emerge from the polyhedral restriction developed in this paper and a similar technique in \cite{Wang2018}. Specifically, observe that the proof of the sufficient conditions for the existence of feasible power flow solutions near a specific operating point in \cite[Theorem 1]{Wang2018} is similar to the proof of Theorem \ref{thm:polyhedral_restriction}. In what follows, we highlight the differences in these two proofs, the specific properties of both sets of sufficient conditions resulting from them, and visually compare these sets. 


In \cite{Wang2018}, the proof does not consider the complete feasibility region as in \eqref{eq:OPF-adjusted-problem}, but manage to formulate the power flow equations and the voltage drop constraint into a similar set constraints as in \eqref{eq:OPF-con-pfe-new}--\eqref{eq:OPF-volt-con-new}. On the contrary, in the proof of Theorem \ref{thm:polyhedral_restriction}, we consider the complete set of equations in \eqref{eq:OPF-adjusted-problem}, assuming that \eqref{eq:OPF-load-con-new}--\eqref{eq:OPF-load-positive-new} are redundant.

The authors in \cite{Wang2018} present explicit sufficient conditions ensuring the existence and uniqueness of solutions to the power flow equations. To do so, they reformulated the power-flow equations as a fixed-point equation that act as a contraction mapping on a complete metric space. They apply Banach's fixed-point theorem to conclude the existence of a unique fixed-point for this equation. In our case, while we guarantee the existence of the fixed-point by Brouwer's fixed-point theorem, we do not guarantee its uniqueness.

In \cite{Wang2018}, the authors identify sufficient conditions without leveraging them for a convex restriction. In contrast, we use our sufficient conditions to construct a polyhedral restriction.

To show the qualitative differences between the polyhedral restriction and the sufficient conditions in \cite{Wang2018}, we visually have a look at both regions. 
We consider the three-node network model in Section \ref{subsec:models} and analyze the feasibility region $S$ in \eqref{eq:feasibility_region_S}, the convex region $B$ that can be derived from the sufficient conditions in \cite{Wang2018}, and the polyhedral restriction $P$ in \eqref{eq:restriction_P}, all around the feasible point $(V_0\mathbf{1}^N,\mathbf{0})$. In what follows, the regions are visually distinguished by color. The feasibility region $S$ is represented in red, the convex restriction $B$ in yellow, and the polyhedral restriction $P$ is illustrated in blue. The overlap between the polyhedral restriction $P$ and the convex restriction is indicated by a green color.

We take reactive power and reactance into account. We set $V_0 = 1,\quad r_{01}=r_{12}= r = 0.01, \quad x_{01}=x_{12}= x = 0.001$, and $\Delta = 0.1$.
In this case, we set a non-negative power generation with power factor 0.9,  i.e. $p_1=-3, q_1 = -3\frac{\sqrt{1-0.9^2}}{0.9}$.
\begin{figure}[h!]
    \centering
    \begin{subfigure}{0.5\linewidth}
  		\centering
  		\includegraphics[width=\linewidth]{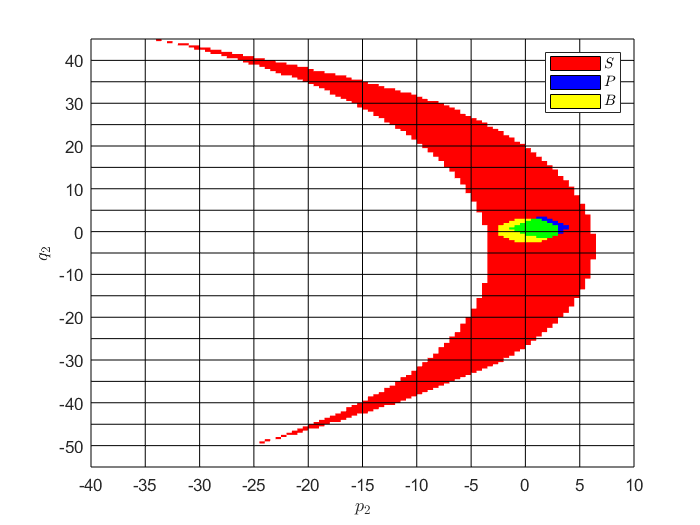}
  		\caption{Wide shot on different regions.}
  	\label{fig:inner_approximation_SP_6_plus_boudec}
	\end{subfigure}%
	\begin{subfigure}{.5\linewidth}
  		\centering
  		\includegraphics[width=\linewidth]{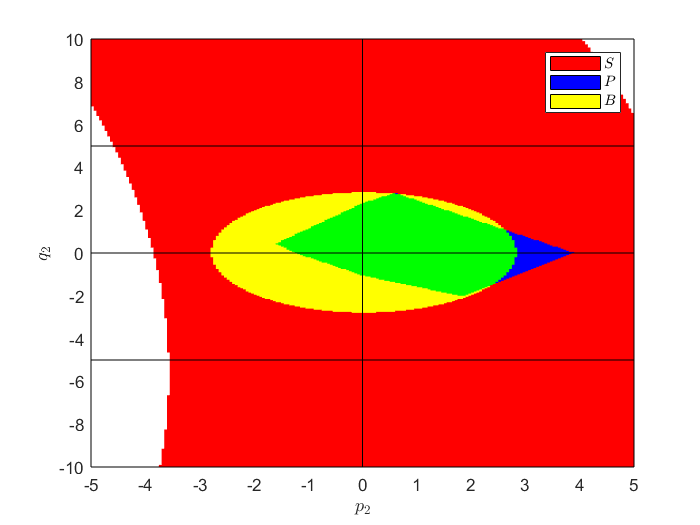}
  		\caption{Zoomed-in shot on regions.}
  		\label{fig:inner_approximation_SP_6_plus_boudec}
	\end{subfigure}
	\caption{Feasible region, polyhedral restriction P, and convex region for three-node model with $r=0.01$, $x=0.001$, $p_1 = -3, q_1 = -3\frac{\sqrt{1-0.9^2}}{0.9}$, and $\Delta = 0.1$.}
\end{figure} In this scenario, the polyhedral restriction $P$ is slightly shifted to the right of the convex region $B$, with overlapping and comparable sizes. In conclusion, the methods are comparable but qualitatively different. 

\subsection{Polyhedral restriction in optimization problems}\label{subsec:poly_restriction}

Here, we show the efficacy of Algorithm \ref{alg:restrictions} on the three networks described in Section \ref{subsec:models}. We start with a known initial solution, iteratively construct a restriction, solve the OPF with simplified affine constraints, and repeat the process until convergence, i.e., if the difference in objective values between sequentially found solutions is below some threshold. While optimal solutions are known for the first two networks, the optimal solution for the third network is unknown.

\subsubsection{Two-node example}
We study the two-node network 
analytically. It is known that, for this simple system, the SOCP relaxation is not always exact \cite{Kocuk2016}. Therefore, we construct an example where the SOCP relaxation is no longer exact by imposing stricter voltage drop constraints \cite{Low2014c}.

Without any voltage drop constraints, the feasible region reduces to the following equalities:
\begin{align}
|V_1|^2 - |V_0|^2 & = 2(rp_1+xq_1)-|z|^2\ell,  \label{eq:OPF_two_bus_voltages}\\
	-q_1 & = q_0-x\ell, \label{eq:OPF_two_bus_reactive_power}\\
	-p_1 & = p_0-r\ell, \label{eq:OPF_two_bus_active_power}\\
	\ell & = p_0^2 + q_0^2 \label{eq:OPF_two_bus_current}
\end{align} If $s_1$ is a controllable load with a given power factor, the variables are $(p_0,q_0,|V_1|^2,\ell)$, and the feasible set consists of solutions to \eqref{eq:OPF_two_bus_voltages}--\eqref{eq:OPF_two_bus_current}.  Substituting \eqref{eq:OPF_two_bus_reactive_power}--\eqref{eq:OPF_two_bus_active_power} into \eqref{eq:OPF_two_bus_current} yields a second-order equation in terms of the variable $\ell$. Solving this equation gives two solutions for $\ell$ corresponding to high- and low-voltage solutions $|V_1|^2$ (according to \eqref{eq:OPF_two_bus_voltages}). In other words, the feasible region consists of the two points of intersection between the line defined by \eqref{eq:OPF_two_bus_reactive_power} and \eqref{eq:OPF_two_bus_active_power} with the convex surface defined by \eqref{eq:OPF_two_bus_current}, making it non-convex.

The SOCP relaxation replaces the equality in \eqref{eq:OPF_two_bus_current} by an inequality, requiring $\ell \geq p_0^2+q_0^2$.
Therefore, the relaxation includes the interior of the convex surface and enlarges the feasible set to the line segment joining these two points.

If the objective function of the optimization problem is linear in the active power $p_0$, then the optimal point over the SOCP feasible set is the lower feasible point, corresponding to the high-voltage solution, and hence, the relaxation is exact.

However, adding voltage drop constraints can result in having a relaxation that is no longer exact. This can be observed as follows. Rewriting \eqref{eq:OPF_two_bus_voltages} in terms of $\ell$, using a fixed voltage magnitude $|V_0|=1$ gives
\begin{align}
\ell = 
\frac{1}{z^2}\left(1 + 2(rp_1+xq_1)-|V_1|^2\right).\label{eq:OPF_two_bus_current_rewritten}
\end{align} Combining \eqref{eq:OPF_two_bus_current_rewritten} and the voltage drop constraints $(1-\Delta)^2 \leq |V_1|^2 \leq (1+\Delta)^2$ gives a box constraint on the variable $\ell$ as
\begin{multline}
\frac{1}{z^2}\left(1 + 2(rp_1+xq_1)-(1+\Delta)^2\right)\leq \ell \leq \\ \leq \frac{1}{z^2}\left(1 + 2(rp_1+xq_1)-(1-\Delta)^2\right).\label{eq:OPF_box_constraint_current}
\end{multline} If constraints \eqref{eq:OPF_box_constraint_current} exclude the lower point, the relaxation is no longer exact, allowing us to compute the optimality gap.

In this example, bus $0$ is the main feeder and has fixed voltage magnitude $|V_0|=1$. The load bus consumes $0.1$ real power with a power factor of $10/\sqrt{101}\approx 1$, resulting in a reactive power consumption of $0.01$. 
The line impedance is defined by $z=0.7+0.1\mathrm{i}$. Recall that the objective is to minimize the power generation $p_0$.

Now, the OPF feasible region is summarized as follows:
\begin{table}[h!]
\centering
\begin{tabular}{lllll}
\hline
                      & $\ell$ & $|V_1|^2$ & $p_0$ & $q_0$ \\ \hline
High-voltage solution & 0.0089 & 1.1376 & -0.0938 & -0.0091 \\
Low-voltage solution  & 2.2751 & 0.0044 & 1.4926 & 0.2175     \\ \hline
\end{tabular}
\caption{Feasible region of the specific two-node example.}
\label{tab:case_delta_0_1}
\end{table}

Formally, the solution for the low-voltage solution meets all constraints. However, it implies consumption instead of generation at the main feeder, lacking a physical interpretation.

As discussed before, the feasibility of solutions in the OPF and SOCP depends on the parameter $\Delta$. We examine two cases: $\Delta = 0.1$ corresponds to voltage drop constraints, $0.81 \leq |V_1|^2 \leq 1.21$, while $\Delta = 0.05$ corresponds to constraints of $0.9025 \leq |V_1|^2 \leq 1.1025$.

\begin{itemize}[wide,nosep]
\item Suppose $\Delta=0.1$. In our example, $r=0.7, x=0.1, p_1=0.1, q_1 = 0.01$ are given. Using \eqref{eq:OPF_box_constraint_current} yields that $-0.1360 \leq \ell \leq 0.6640$. Therefore, only the high-voltage solution, as depicted in Table \ref{tab:case_delta_0_1}, is feasible. The high-voltage solution is in the feasible set of SOCP and, hence, the relaxation is exact.
\item Suppose $\Delta=0.05$. Using \eqref{eq:OPF_box_constraint_current} yields that $0.0790\leq \ell \leq 0.4790$. Therefore, also the high-voltage solution is no longer feasible and the relaxation is not exact. Instead, the SOCP relaxation obtains an infeasible solution, namely at the intersection of the line defined by \eqref{eq:OPF_two_bus_reactive_power}--\eqref{eq:OPF_two_bus_active_power} and the lower bound in \eqref{eq:OPF_box_constraint_current} which yields $p_0 = -0.0447$ with an optimality gap of $-0.0937+0.0447 = -0.0490$.
\end{itemize}

In this specific example, using the polyhedral restriction of Theorem \ref{thm:polyhedral_restriction} yields the following optimization problem:
\begin{subequations}
\begin{align}
     \min_{p_1^c,q_1^c,p_1^g,q_1^g} &~ p_0^c-p_0^g \nonumber \\
	 \quad \textrm{s.t.} \quad ~ &(\mathbf{A}+\Delta\mathbf{B})\begin{pmatrix}
	p_1^c, q_1^c, p_1^g, q_1^g
\end{pmatrix}^T \leq \nonumber  \\
& \quad \Delta\mathbf{1}^{4}+(\mathbf{A}-\Delta(\mathbf{B}+(1-\Delta)\mathbf{C}))\begin{pmatrix}
0, 0, 0.1, 0.01
\end{pmatrix}^T \nonumber \\
	 &  p_1^c = 0,\quad p_1^g= 0.1,
	 \quad q_1^c = 0,\quad q_1^g= 0.01, \label{eq:constraints_active_reactive_poly}
\end{align}\label{eq:polyhedral_restriction_two_node}
\end{subequations}
where the matrices $\mathbf{A}, \mathbf{B}$ and $\mathbf{C}$ are given by \eqref{eq:A}--\eqref{eq:C}.

For $\Delta=0.1$, the only feasible point of \eqref{eq:polyhedral_restriction_two_node} is $s_1 = 0.1+0.01\mathrm{i}$, aligning with constraints \eqref{eq:constraints_active_reactive_poly}. This immediately yields the power generation at the main feeder by \eqref{eq:power_root_node} as $s_0 = -0.0938-0.0109\mathrm{i}$, resulting in $p_0 = -0.0938$. However, for $\Delta=0.05$, there is no feasible point for \eqref{eq:polyhedral_restriction_two_node}.

\subsubsection{Three-node example}
In this example, we consider the three-node model from Section \ref{subsec:models}. We set the voltage at the root node to $V_0=1$, and use equal resistances on the lines, $r_{01}=r_{12}=0.01$, and equal reactances, $x_{01}=x_{12}=0.001$. This yields a relatively high resistance-to-reactance ratio, which is usual in distribution networks. Additionally, we define the voltage drop control parameter as $\Delta=0.1$, and we constrain power consumption and generation within the bounds $\mathcal{P}_1^c = \mathcal{P}_2^c = [0,35]$ and $
\mathcal{P}_1^g = \mathcal{P}_2^g = [0,35]$.


We address two optimization problems over these regions: one maximizes the active load, and the other minimizes the active generation in the network using Algorithm \ref{alg:restrictions}.

Irrespective of the objective function, Algorithm \ref{alg:restrictions} follows the same steps. For initialization, we choose the feasible point $(\mathbf{V},\mathbf{s}) = (\mathbf{1},\mathbf{0})$, setting $\mathbf{V}^{(0)}=\mathbf{1}$ and $\mathbf{s}^{(0)}=\mathbf{0}$. We ensure a maximum voltage deviation of $10\%$ from the nominal voltage magnitude $|\hat{V}_0|=1$. This means that all voltage vectors that satisfy the constraint $|V_j^{(k)}-\mathbf{1}|\leq 0.1$ for all $j\in\mathcal{N}\backslash\{0\}$ and every iteration $k$ are allowed. In other words, we set $\Delta^{(0)} = 0.1$. In the first step, we construct a polyhedral restriction around the feasible point $(\mathbf{V}^{(0)},\mathbf{s}^{(0)})$ using Theorem \ref{thm:polyhedral_restriction} (or in the special case of $(\mathbf{V},\mathbf{s}) = (\mathbf{1},\mathbf{0})$ using Corollary \ref{cor:polyhedral_restriction}). For the second step, we solve the optimization problem defined in \eqref{eq:OPF_prime_P} over the constructed polyhedral restriction. The outcome, a new load vector $\mathbf{s}^{(1)}$, maximizes the objective function in \eqref{eq:OPF_prime_P} and is contained in the true feasible region. As a third step, the new load vector $\mathbf{s}^{(1)}$ is used to compute the corresponding voltage vector $\mathbf{V}^{(1)}$ via the power flow equations in \eqref{eq:fixed_point_eq}. To enforce the constraint $|V_j^{(2)}-\hat{V}_0|\leq \Delta^{(0)}$ for all $j\in\mathcal{N}\backslash\{0\}$ for the second iteration, we allow the maximal difference between corresponding elements of the vectors $\mathbf{V}^{(2)}$ and $\mathbf{V}^{(1)}$, 
to be $\Delta^{(1)}:=\Delta^{(0)}-\|\mathbf{V}^{(1)}-\mathbf{1}\|_{\infty}$,
such that,
\begin{align}
\|\mathbf{V}^{(2)}-\mathbf{1}\|_{\infty} & \leq \|\mathbf{V}^{(2)}-\mathbf{V}^{(1)}\|_{\infty} + \|\mathbf{V}^{(1)}-\mathbf{1}\| \label{eq:desired_inequality_delta_1}\\
& \leq \Delta^{(1)} + \|\mathbf{V}^{(1)}-\mathbf{1}\|_{\infty} \leq \Delta^{(0)}, \label{eq:desired_inequality_delta}
\end{align} as desired. For any iteration $k$, updating the parameter $\Delta^{(k)}$ to control the voltage drop as
\begin{align*}
\Delta^{(k)}:=\Delta^{(0)}-\|\mathbf{V}^{(k)}-\mathbf{1}\|_{\infty},
\end{align*} yields the desired inequality $|V_j^{(k)}-1|\leq \Delta^{(0)}$ for all nodes $j\in\mathcal{N}\backslash\{0\}$ and every iteration $k$, according to the same logic as in \eqref{eq:desired_inequality_delta_1}--\eqref{eq:desired_inequality_delta} for any iterations $k$ and $k+1$ instead of iterations $1$ and $2$. Iterating this procedure of constructing polyhedral restrictions around feasible points, solving optimization problems with polyhedral constraints, and updating new optimal feasible points while the relative Eucledian distance between subsequent load flow solutions $\mathbf{s}^{(k+1)}$ and $\mathbf{s}^{(k)}$ is larger than a predefined threshold $\epsilon = 0.01$, yield a sequence of optimal (with respect to each polyhedral restriction) and feasible load flow solutions $\mathbf{s}^{(1)},\ldots,\mathbf{s}^{(K)}$.

We have discussed the sequential optimization method for a general objective function in the three-node network. Now, we discuss the case where we maximize the active load.

\paragraph{Maximize active load} We aim to solve the following OPF problem, where the objective is to maximize the active load on the network, subject to several constraints:
\begin{align}
    \max_{\mathbf{p}^c,\mathbf{p}^g} & \ \sum_{j=1}^N (p_j^c-p_j^g) \quad
	\textrm{s.t.} \quad \eqref{eq:OPF-con-pfe}\text{--}\eqref{eq:OPF-load-con} . \label{eq:OPF_prime}
\end{align}

For the polyhedral restriction, we substitute the constraints in \eqref{eq:OPF-problem} by the constraints in \eqref{eq:OPF-adjusted-problem}, to obtain
\begin{align}
    \max_{\mathbf{p}^c,\mathbf{p}^g} \ \sum_{j=1}^N (p_j^c-p_j^g)\quad
	\textrm{s.t.} \quad & \eqref{eq:OPF-con-pfe-new}\text{--}\eqref{eq:OPF-load-positive-new}.\label{eq:OPF_prime_P}
\end{align}

We compare the optimal solutions of the convex relaxation \eqref{eq:OPF_prime} and the polyhedral restriction of \eqref{eq:OPF_prime_P}. In this example, the convex relaxation is exact \cite[Theorem 4]{Gan2015b}, meaning the optimal solution of the convex relaxation matches the true optimal solution. Meanwhile, the optimal solution of the polyhedral restriction serves a lower bound on the true optimal solution.

\begin{figure}
    \centering
    \includegraphics[width=.24\textwidth]{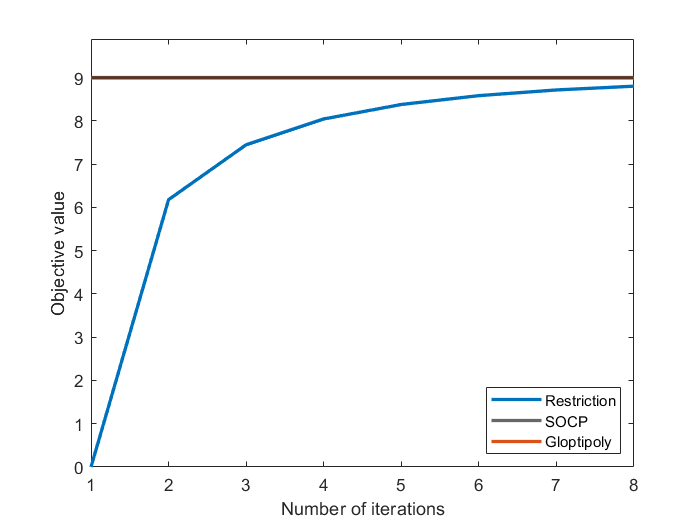}
    \includegraphics[width=.24\textwidth]{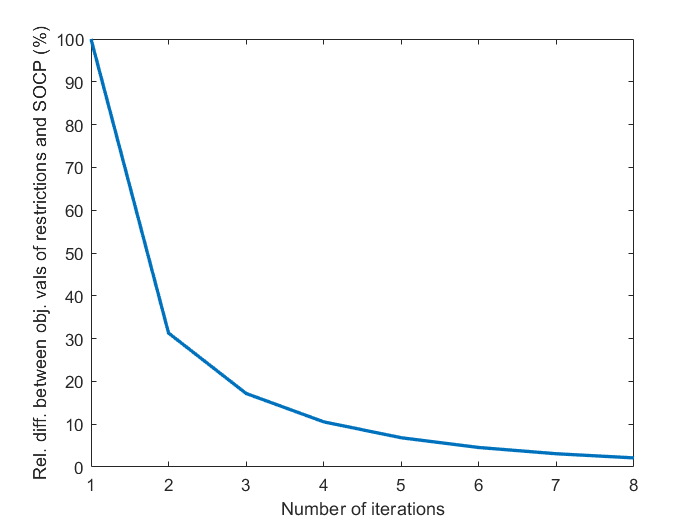}
    \caption{Accuracy of \eqref{eq:OPF_prime_P} for the three-node model.}
    \label{fig:accuracy_of_opf_prime_three_node_model}
\end{figure}

In Figure \ref{fig:accuracy_of_opf_prime_three_node_model}, the left side displays the optimal values of \eqref{eq:OPF_prime} and \eqref{eq:OPF_prime_P} as a function of the iteration count. Notably, the optimal values of the polyhedral restrictions increase with iterations, serving as a lower bound for the optimal value of \eqref{eq:OPF_prime}. On the right side of Figure \ref{fig:accuracy_of_opf_prime_three_node_model}, the relative error between the optimal values of \eqref{eq:OPF_prime} and \eqref{eq:OPF_prime_P} is shown. The relative error remains below $2\%$ after the sequential optimization method reached its predefined threshold.

So far, we have discussed the case where the resistance-to-reactance ratio is the same throughout the network, and the resistance and reactance of each cable are identical. Below, we extend our discussion for a diverse range of the resistance-to-reactance ratio, varying between the two lines. Given the typically high resistance-to-reactance ratio in distribution networks, we set the reactance values at $x_{01} = x_{12} = 0.001$ and vary the resistance values between $0.007$ and $0.03$, resulting in a resistance-to-reactance ratio variation between $7$ and $30$.
By varying the resistance values, 
we obtain that the optimal values by the sequential optimization method are close to the true optimal values found by the SOCP relaxation. The maximum relative difference between the objective values is approximately $2\%$, because of the predefined threshold $\epsilon=0.01$ in Algorithm \ref{alg:restrictions}. For example, when we set $\epsilon=0.001$, the maximum relative difference across the range of resistance and reactance values decreases to $0.25\%$. The procedure appears insensitive to the resistance-to-reactance ratio on the line that is furthest away from the main feeder. This is due to the marginal effect of the optimal power load on the second node; to maximize the total power load in the network, it is optimal to put approximately no load on the second node.

\paragraph{Minimize active load} In this example, we aim to solve the OPF problem in \eqref{eq:OPF_prime} as before, but where the objective is to minimize (instead of maximize) the active load on the network, subject to the same constraints.

Similarly, for the polyhedral restriction, we substitute the constraints in \eqref{eq:OPF-problem} by the constraints in \eqref{eq:OPF-adjusted-problem}, and change the objective function from a maximizing to minimizing problem.




In Figure \ref{fig:iterations_three_node_case_obj_vals}, the objective values according to the iterative process are also visualized, next to the objective value found by the SOCP relaxation and the true solution found by solving the optimization problem in \eqref{eq:OPF_prime}. 

Now, similar to the maximization case, we discuss similar results for a diverse range of the resistance over reactance ratios. We fix the reactance values at $x_{01} = x_{12} = 0.001$ and vary the resistance values between $0.007$ and $0.03$, such that the resistance over reactance ratio varies between $7$ and $30$.


\begin{figure}
    \centering
    \includegraphics[width=.6\linewidth]{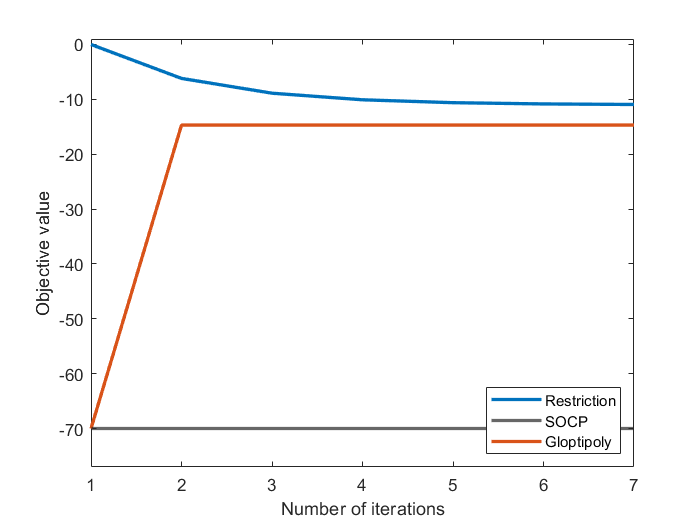}
    \caption{Objective values found by the polyhedral restrictions, the SOCP relaxation and GloptiPoly for different iterations.}
    \label{fig:iterations_three_node_case_obj_vals}
\end{figure}

By varying the resistance values, 
we see that relative difference between the optimal values obtained by the procedure of iteratively constructing polyhedral restrictions around feasible points and the true optimal values obtained by GloptiPoly is large. The maximum relative difference between the objective values is approximately $30\%$ when the predefined threshold $\epsilon=0.01$ in Algorithm \ref{alg:restrictions} is. Of course, when we set $\epsilon=0.001$, the maximum relative difference decreases for the complete range of resistance and reactance values, but not as much as in the maximization case. Also in this case, the maximum relative difference is around $30\%$.

\subsubsection{SCE-47 network} We consider the SCE-47 network with $|V_0|=1$ and set the parameter to control the voltage drop at $\Delta=0.1$. We distinguish between two different types of customers in the network. We assume that the set of nodes that can only generate electricity is given by $\mathcal{S} = \{12,16,18,22,23\}$. The other nodes can generate and consume electricity. Therefore, we range the bound on the power consumption and generation for $\overline{p_j^c}=\overline{p_j^g}$ between $0.005$ and $0.03$ for $j\in\mathcal{N}\backslash\{0\}$ and we set the bound on the power generation at $\overline{p_j^g}=0.01$, and maximize and minimize the active load in the network.


\begin{figure}[h!]
\centering
\begin{subfigure}{.25\textwidth}
  \centering
  \includegraphics[width=\linewidth]{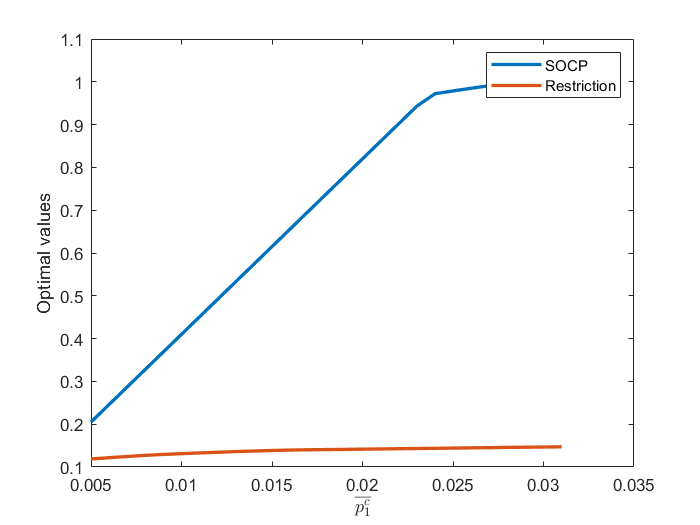}
  \caption{Maximization of active load.}
  \label{fig:optimal_point_SDP_P_sce47_max}
\end{subfigure}%
\begin{subfigure}{.25\textwidth}
  \centering
  \includegraphics[width=\linewidth]{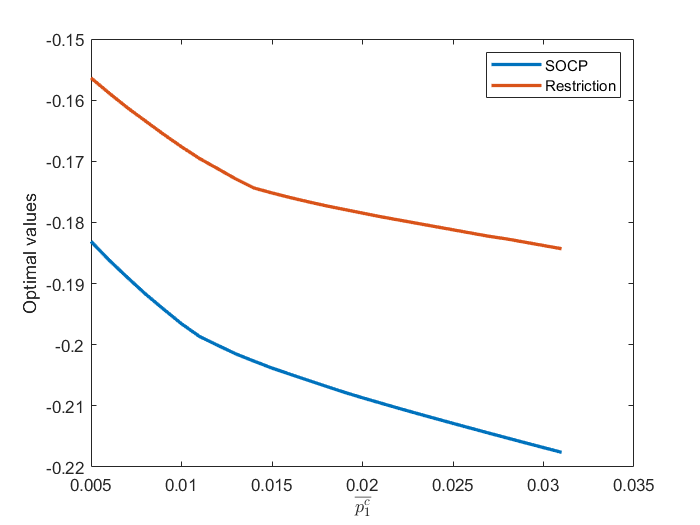}
  \caption{Minimization of active load.}
 \label{fig:optimal_point_SDP_P_sce47_min}
\end{subfigure}
\caption{Solutions of the SOCP relaxation and the final iteration of the polyhedral restrictions.}
\label{fig:accuracy_opf_prime_P_sce47}
\end{figure}

The results are shown in Figure \ref{fig:accuracy_opf_prime_P_sce47}. The maximal netto active load is different for the SOCP relaxation and for the polyhedral restriction. The relative error between the solutions of the SOCP relaxation and the polyhedral restriction is larger in the case of maximization of the active load than in the case of minimization. However, the solutions obtained by the SOCP relaxation are infeasible, i.e., the solutions do not satisfy the power flow equations.
The solutions obtained by the final iteration of the polyhedral restriction are feasible and provide at least bounds on the true optimal solutions.

\section{Conclusion}\label{sec:conclusion}
We have shown how to construct a polyhedral restriction of the feasibility region. The restriction can be built given network parameters, such as the topology and admittance matrix, and a feasible operating point.
We have proposed a sequential optimization method to compute (bounds on) solutions to OPF problems, such that we can guarantee feasibility of the final solution. The final solution can always be used as a bound on the true optimal solution.
Empirical studies have verified that the polyhedral restriction can be used to assess whether there exists a feasible solution that satisfies the power flow equations and voltage drop constraint, and to compute (bounds on) optimal power flow solutions for several test networks.

\appendix

In this section, we show that under the conditions of Theorem \ref{thm:polyhedral_restriction}, operator $G$ is a self-map of the voltages $\mathbf{V}$ and is continuous on the compact set $D$. To show that $G$ is a self-map, we need additional Lemmas \ref{lemma:ineq_voltage_magnitude} and \ref{lemma:ineq_voltage_angles}.

\subsection{Proof of Lemma \ref{lemma:self-mapping}}

\begin{proof}
Let $\tilde{\mathbf{s}}\in P$ and $\mathbf{V}\in D$.  Since $(\hat{\mathbf{V}},\hat{\mathbf{s}})$ satisfies the power flow equations in \eqref{eq:fixed_point_eq}, we have that
\begin{align}
    \hat{\mathbf{V}} = G(\hat{\mathbf{V}}) = V_0\mathbf{1}^N-\mathbf{Z}\text{diag}(\hat{\mathbf{V}}^*)^{-1}\hat{\mathbf{s}}^*.\label{eq:solution_satisfies_PFE}
\end{align} We have to show that $G(\mathbf{V})\in D$, i.e., or equivalently
\begin{align}
    \|G(\mathbf{V})-\hat{\mathbf{V}}\|_{\infty} \leq \Delta|\hat{\mathbf{V}}|.\label{eq:equivalent_self_mapping}
\end{align} Plugging in \eqref{eq:solution_satisfies_PFE} in the left-hand side of \eqref{eq:equivalent_self_mapping} yields
\begin{multline*}
    \|G(\mathbf{V})-\hat{\mathbf{V}}\|_{\infty} 
    = \\ = \|\mathbf{Z}(\text{diag}(\hat{\mathbf{V}}^*)^{-1}-\text{diag}(\mathbf{V}^*)^{-1})\mathbf{\hat{s}}^* + \mathbf{Z}\text{diag}(\mathbf{V}^*)^{-1}(\mathbf{\hat{s}}^*-\mathbf{s}^*)\|_{\infty}.\nonumber
\end{multline*} 
Continuing the derivation, we have by the triangle inequality that
\begin{multline}
    \|G(\mathbf{V})-\hat{\mathbf{V}}\|_{\infty} 
    \leq \\ \leq \bigg\|\mathbf{Z}\left(\text{diag}\left(\frac{\mathbf{V}^*-\hat{\mathbf{V}}^*}{\hat{\mathbf{V}}^*\mathbf{V}^*} \right) \right)\mathbf{\hat{s}}^* \bigg\|_{\infty} + \bigg\|\mathbf{Z}\text{diag}(\mathbf{V}^*)^{-1}(\mathbf{\hat{s}}^*-\mathbf{s}^*)\bigg\|_{\infty}.\label{eq:both_terms}
\end{multline} Consider both terms in \eqref{eq:both_terms} separately. For the first term in \eqref{eq:both_terms}, we have by the definition of the vector infinity norm and the use of the triangle inequality,
\begin{multline}
    \bigg\|\mathbf{Z}\left(\text{diag}\left(\frac{\mathbf{V}^*-\hat{\mathbf{V}}^*}{\hat{\mathbf{V}}^*\mathbf{V}^*} \right) \right)\mathbf{\hat{s}}^* \bigg\|_{\infty} \leq \\
     \leq \max_{j\in\mathcal{N}\backslash\{0\}}\sum_{k=1}^N \big|Z_{jk}\big|\big|\hat{s}_k^*\big|\frac{|V_k^*-\hat{V}_k^*|}{|\hat{V}_k^*||V_k^*|}.\label{eq:first_term} 
\end{multline} Since $\mathbf{V}\in D$, we have $|V_j^*-\hat{V}_j^*|=|V_j-\hat{V}_j|\leq \Delta|\hat{V}_j|$ for all $j\in\mathcal{N}\backslash\{0\}$. By the triangle inequality, we have $|\hat{s}_k| = |\hat{p}_k+\mathrm{i}\hat{q}_k| \leq |\hat{p}_k|+|\hat{q}_k|$. Moreover, by construction, $|\hat{p}_k| = |\hat{p}_k^c-\hat{p}_k^g| = \hat{p}_k^c+\hat{p}_k^g$, and similarly $|\hat{q}_k| = |\hat{q}_k^c-\hat{q}_k^g| = q_k^c+q_k^g$. Applying these inequalities and equalities in \eqref{eq:first_term} yield
\begin{multline}
    \bigg\|\mathbf{Z}\left(\text{diag}\left(\frac{\mathbf{V}^*-\hat{\mathbf{V}}^*}{\hat{\mathbf{V}}^*\mathbf{V}^*} \right) \right)\mathbf{\hat{s}}^* \bigg\|_{\infty} 
    \leq \\ \leq \max_{j\in\mathcal{N}\backslash\{0\}}\sum_{k=1}^N \frac{\big|Z_{jk}\big|(\hat{p}_k^c+\hat{p}_k^g+\hat{q}_k^c+\hat{q}_k^g)\Delta}{|V_k^*|}.\label{eq:denominator_voltage}
\end{multline} 
Finally, we use Lemma \ref{lemma:ineq_voltage_magnitude} in the denominator of \eqref{eq:denominator_voltage} to bound the first term in \eqref{eq:both_terms} as
\begin{multline}
\bigg\|\mathbf{Z}\left(\text{diag}\left(\frac{\mathbf{V}^*-\hat{\mathbf{V}}^*}{\hat{\mathbf{V}}^*\mathbf{V}^*} \right) \right)\mathbf{\hat{s}}^* \bigg\|_{\infty} \leq \\ \leq \max_{j\in\mathcal{N}\backslash\{0\}}\underbrace{\sum_{k=1}^N \frac{\big|Z_{jk}\big|(\hat{p}_k^c+\hat{p}_k^g+\hat{q}_k^c+\hat{q}_k^g)\Delta}{(1-\Delta)|\hat{V}_k|}}_{:=\beta_1}.\label{eq:first_term_of_both_terms}
\end{multline} For the second term in \eqref{eq:both_terms}, we have by definition
\begin{align*}
    \bigg\|\mathbf{Z}\text{diag}(\mathbf{V}^*)^{-1}(\mathbf{\hat{s}}^*-\mathbf{s}^*)\bigg\|_{\infty} 
    & = \max_{j\in\mathcal{N}\backslash\{0\}}\left|\sum_{k=1}^N Z_{jk}\frac{(\hat{s}_k^*-s_k^*)V_k}{|V_k|^2} \right|.
\end{align*} Using the definition of $V_k$ to express each voltage as $V_k = |V_k|\exp(\mathrm{i}\theta_k)$, and multiplying by $\exp(-\mathrm{i}\hat{\theta}_j)$ in the numerator and denominator yields
\begin{multline}
\bigg\|\mathbf{Z}\text{diag}(\mathbf{V}^*)^{-1}(\mathbf{\hat{s}}^*-\mathbf{s}^*)\bigg\|_{\infty} = \\
     = \max_{j\in\mathcal{N}\backslash\{0\}}\left|\sum_{k=1}^N Z_{jk}\frac{\splitfrac{(\hat{s}_k^*-s_k^*)(|V_k|\cos(\theta_k-\hat{\theta}_j)+}{+\mathrm{i}|V_k|\sin(\theta_k-\hat{\theta}_j)))}}{|V_k|^2} \right|. \label{eq:GV-V}
\end{multline} since we have that $\left|1/\exp(-\mathrm{i}\theta_j)\right| = 1$.

We have $Z_{jk}=R_{jk}+\mathrm{i}X_{jk}$ for all $j\in\mathcal{N}\backslash\{0\}$ and $\hat{s}_k^*-s_k^* = \hat{p}_k-p_k-\mathrm{i}(\hat{q}_k-q_k) = \hat{p}_k^c-\hat{p}_k^g-(p_k^c-p_k^g)-\mathrm{i}(\hat{q}_k^c-\hat{q}_k^g-(q_k^c-q_k^g))$. Using these notations and separating in sine and cosine terms, we rewrite \eqref{eq:GV-V} to
\begin{multline}
    \Biggl| \sum_{k=1}^N \frac{\splitfrac{(R_{jk}(\hat{p}_k^c-\hat{p}_k^g-(p_k^c-p_k^g))+}{+X_{jk}(\hat{q}_k^c-\hat{q}_k^g-(q_k^c-q_k^g)))|V_k|\cos(\theta_k-\hat{\theta}_j)}}{|V_k|^2} + \\
    +\mathrm{i}\frac{\splitfrac{(-R_{jk}(\hat{q}_k^c-\hat{q}_k^g-(q_k^c-q_k^g))+}{+X_{jk}(\hat{p}_k^c-\hat{p}_k^g-(p_k^c-p_k^g)))|V_k|\cos(\theta_k-\hat{\theta}_j)}}{|V_k|^2}  + \\
    + \frac{\splitfrac{(R_{jk}(\hat{q}_k^c-\hat{q}_k^g-(q_k^c-q_k^g))-}{-X_{jk}(\hat{p}_k^c-\hat{p}_k^g-(p_k^c-p_k^g)))|V_k|\sin(\theta_k-\hat{\theta}_j)}}{|V_k|^2} + \\
    + \mathrm{i}\frac{\splitfrac{(R_{jk}(\hat{p}_k^c-\hat{p}_k^g-(p_k^c-p_k^g))+}{+X_{jk}(\hat{q}_k^c-\hat{q}_k^g-(q_k^c-q_k^g)))|V_k|\sin(\theta_k-\hat{\theta}_j)}}{|V_k|^2}\Biggr|\label{eq:bound_voltage_diff_1}
\end{multline} to apply the bounds in Lemma \ref{lemma:ineq_voltage_angles} later. First, we use the triangle inequality to bound the expression in  \eqref{eq:bound_voltage_diff_1} to
\begin{multline}
    \Biggl| \underbrace{\sum_{k=1}^N \frac{\splitfrac{(R_{jk}(\hat{p}_k^c-\hat{p}_k^g-(p_k^c-p_k^g))+}{+X_{jk}(\hat{q}_k^c-\hat{q}_k^g-(q_k^c-q_k^g)))|V_k|\cos(\theta_k-\hat{\theta}_j)}}{|V_k|^2}}_{:=\alpha_1} \Biggr| + \\
    +\Biggl| \underbrace{\sum_{k=1}^N \frac{\splitfrac{(-R_{jk}(\hat{q}_k^c-\hat{q}_k^g-(q_k^c-q_k^g))+}{+X_{jk}(\hat{p}_k^c-\hat{p}_k^g-(p_k^c-p_k^g)))|V_k|\cos(\theta_k-\hat{\theta}_j)}}{|V_k|^2}}_{:=\alpha_2} \Biggr| + \\
    + \Biggl|\underbrace{\sum_{k=1}^N \frac{\splitfrac{((\mathrm{i}R_{jk}-X_{jk})(\hat{p}_k^c-\hat{p}_k^g-(p_k^c-p_k^g))+}{\splitfrac{+(R_{jk}+\mathrm{i}X_{jk})(\hat{q}_k^c-\hat{q}_k^g-(q_k^c-q_k^g))))}{|V_k|\sin(\theta_k-\hat{\theta}_j)}}}{|V_k|^2}}_{:=\beta}\Biggr|.\label{eq:alpha_beta}
\end{multline} Then, again by using the triangle inequality twice, we bound the third term in \eqref{eq:alpha_beta}, i.e. $\beta$, by
\begin{align}
|\beta| 
& \leq \sum_{k=1}^N \frac{\splitfrac{(|\mathrm{i}R_{jk}-X_{jk}||\hat{p}_k^c-\hat{p}_k^g-(p_k^c-p_k^g)|+}{\splitfrac{+|R_{jk}+\mathrm{i}X_{jk}||\hat{q}_k^c-\hat{q}_k^g-(q_k^c-q_k^g)||)}{|V_k|\sin(\theta_k-\hat{\theta}_j)|}}}{|V_k|^2}.
\end{align} Furthermore, notice that by definition, we have the equality $|\mathrm{i}R_{jk}-X_{jk}|=|R_{jk}+\mathrm{i}X_{jk}| = |Z_{jk}|$, and the inequalities $|\hat{p}_k-p_k|\leq |\hat{p}_k|+|p_k|$, and similarly $|\hat{q}_k-q_k| \leq |\hat{q}_k|+|q_k|$ by the triangle inequality. Thus,
\begin{align}
|\beta|\leq \sum_{k=1}^N \frac{|Z_{jk}|(|\hat{p}_k|+|p_k|+|\hat{q}_k|+|q_k|)||V_k|\sin(\theta_k-\hat{\theta}_j)|}{|V_k|^2}.\label{eq:beta2}
\end{align} In what follows, we use the following equalities: $|p_k|=|p_k^c-p_k^g|= p_k^c+p_k^g$ and similarly, $|q_k| = |q_k^c-q_k^g| = q_k^c+q_k^g$. We justify these in Remark \ref{remark:absolute_values} later on. From these equalities and applying Lemma \ref{lemma:ineq_voltage_angles} in \eqref{eq:beta2} then yields
\begin{align}
|\beta|\leq \underbrace{\sum_{k=1}^N \frac{|Z_{jk}|(p_k^c+q_k^c+p_k^g+q_k^g+\hat{p}_k^g+\hat{q}_k^c+\hat{p}_k^g+\hat{q}_k^g)\Delta}{(1-\Delta)^2|\hat{V}_k|}}_{:=\beta_2}.\label{eq:beta2_1}
\end{align} Now, we focus on the terms $\alpha_1$ and $\alpha_2$ in \eqref{eq:alpha_beta}. By definition of absolute values, we have $ |\alpha_1|+|\alpha_2| = \max(\alpha_1,-\alpha_1)+\max(\alpha_2,-\alpha_2)$. Then, by case distinction, we have
\begin{multline}
    |\alpha_1|+|\alpha_2| \\ = \max(\alpha_1+\alpha_2,\alpha_1-\alpha_2,-\alpha_1+\alpha_2,-\alpha_1-\alpha_2),
    \label{eq:absolute_value_sum}
\end{multline} 
so it suffices to bound each of the four expressions in \eqref{eq:absolute_value_sum}.
We have $\mathbf{V}\in D$ by assumption, so we can use the bounds from Lemmas \ref{lemma:ineq_voltage_magnitude} and \ref{lemma:ineq_voltage_angles}. Furthermore, recall that $p_k^c,q_k^c,p_k^g,q_k^g\geq 0$ for all $k\in\mathcal{N}\backslash\{0\}$. Applying these bounds gives
\begin{subequations}\label{eq:bounds_alpha_1}
\begin{multline}
\alpha_1 \leq \sum_{k=1}^N \frac{-R_{jk}(p_k^c-\hat{p}_k^c)}{(1-\Delta)^2|\hat{V}_k|^2}+\frac{\Delta R_{jk}(p_k^c-\hat{p}_k^c)}{(1-\Delta)^2|\hat{V}_k|^2}  + \frac{2\Delta R_{jk}\hat{p}_k^c}{(1-\Delta)^2|\hat{V}_k|^2} \\
\frac{-X_{jk}(q_k^c-\hat{q}_k^c)}{(1-\Delta)^2|\hat{V}_k|^2}+\frac{\Delta X_{jk}(q_k^c-\hat{q}_k^c)}{(1-\Delta)^2|\hat{V}_k|^2}  + \frac{2\Delta X_{jk}\hat{q}_k^c}{(1-\Delta)^2|\hat{V}_k|^2} \\
\frac{R_{jk}(p_k^g-\hat{p}_k^g)}{(1-\Delta)^2|\hat{V}_k|^2}+\frac{\Delta R_{jk}(p_k^g-\hat{p}_k^g)}{(1-\Delta)^2|\hat{V}_k|^2}  + \frac{2\Delta R_{jk}\hat{p}_k^g}{(1-\Delta)^2|\hat{V}_k|^2} \\
\frac{X_{jk}(q_k^g-\hat{q}_k^g)}{(1-\Delta)^2|\hat{V}_k|^2}+\frac{\Delta X_{jk}(q_k^g-\hat{q}_k^g)}{(1-\Delta)^2|\hat{V}_k|^2}  + \frac{2\Delta X_{jk}\hat{q}_k^g}{(1-\Delta)^2|\hat{V}_k|^2},
\end{multline} and
\begin{multline}
-\alpha_1 \leq \sum_{k=1}^N \frac{R_{jk}(p_k^c-\hat{p}_k^c)}{(1-\Delta)^2|\hat{V}_k|^2}+\frac{\Delta R_{jk}(p_k^c-\hat{p}_k^c)}{(1-\Delta)^2|\hat{V}_k|^2}  + \frac{2\Delta R_{jk}\hat{p}_k^c}{(1-\Delta)^2|\hat{V}_k|^2} \\
\frac{X_{jk}(q_k^c-\hat{q}_k^c)}{(1-\Delta)^2|\hat{V}_k|^2}+\frac{\Delta X_{jk}(q_k^c-\hat{q}_k^c)}{(1-\Delta)^2|\hat{V}_k|^2}  + \frac{2\Delta X_{jk}\hat{q}_k^c}{(1-\Delta)^2|\hat{V}_k|^2} \\
\frac{-R_{jk}(p_k^g-\hat{p}_k^g)}{(1-\Delta)^2|\hat{V}_k|^2}+\frac{\Delta R_{jk}(p_k^g-\hat{p}_k^g)}{(1-\Delta)^2|\hat{V}_k|^2}  + \frac{2\Delta R_{jk}\hat{p}_k^g}{(1-\Delta)^2|\hat{V}_k|^2} \\
\frac{-X_{jk}(q_k^g-\hat{q}_k^g)}{(1-\Delta)^2|\hat{V}_k|^2}+\frac{\Delta X_{jk}(q_k^g-\hat{q}_k^g)}{(1-\Delta)^2|\hat{V}_k|^2}  + \frac{2\Delta X_{jk}\hat{q}_k^g}{(1-\Delta)^2|\hat{V}_k|^2},
\end{multline}
\end{subequations} Similarly,
\begin{subequations}\label{eq:bounds_alpha_2}
\begin{multline}
\alpha_2 \leq \sum_{k=1}^N \frac{-X_{jk}(p_k^c-\hat{p}_k^c)}{(1-\Delta)^2|\hat{V}_k|^2}+\frac{\Delta X_{jk}(p_k^c-\hat{p}_k^c)}{(1-\Delta)^2|\hat{V}_k|^2}  + \frac{2\Delta X_{jk}\hat{p}_k^c}{(1-\Delta)^2|\hat{V}_k|^2} \\
\frac{R_{jk}(q_k^c-\hat{q}_k^c)}{(1-\Delta)^2|\hat{V}_k|^2}+\frac{\Delta R_{jk}(q_k^c-\hat{q}_k^c)}{(1-\Delta)^2|\hat{V}_k|^2}  + \frac{2\Delta R_{jk}\hat{q}_k^c}{(1-\Delta)^2|\hat{V}_k|^2} \\
\frac{X_{jk}(p_k^g-\hat{p}_k^g)}{(1-\Delta)^2|\hat{V}_k|^2}+\frac{\Delta X_{jk}(p_k^g-\hat{p}_k^g)}{(1-\Delta)^2|\hat{V}_k|^2}  + \frac{2\Delta X_{jk}\hat{p}_k^g}{(1-\Delta)^2|\hat{V}_k|^2} \\
\frac{-R_{jk}(q_k^g-\hat{q}_k^g)}{(1-\Delta)^2|\hat{V}_k|^2}+\frac{\Delta R_{jk}(q_k^g-\hat{q}_k^g)}{(1-\Delta)^2|\hat{V}_k|^2}  + \frac{2\Delta R_{jk}\hat{q}_k^g}{(1-\Delta)^2|\hat{V}_k|^2},
\end{multline} and
\begin{multline}
-\alpha_2 \leq \sum_{k=1}^N \frac{X_{jk}(p_k^c-\hat{p}_k^c)}{(1-\Delta)^2|\hat{V}_k|^2}+\frac{\Delta X_{jk}(p_k^c-\hat{p}_k^c)}{(1-\Delta)^2|\hat{V}_k|^2}  + \frac{2\Delta X_{jk}\hat{p}_k^c}{(1-\Delta)^2|\hat{V}_k|^2} \\
\frac{-R_{jk}(q_k^c-\hat{q}_k^c)}{(1-\Delta)^2|\hat{V}_k|^2}+\frac{\Delta R_{jk}(q_k^c-\hat{q}_k^c)}{(1-\Delta)^2|\hat{V}_k|^2}  + \frac{2\Delta R_{jk}\hat{q}_k^c}{(1-\Delta)^2|\hat{V}_k|^2} \\
\frac{-X_{jk}(p_k^g-\hat{p}_k^g)}{(1-\Delta)^2|\hat{V}_k|^2}+\frac{\Delta X_{jk}(p_k^g-\hat{p}_k^g)}{(1-\Delta)^2|\hat{V}_k|^2}  + \frac{2\Delta X_{jk}\hat{p}_k^g}{(1-\Delta)^2|\hat{V}_k|^2} \\
\frac{R_{jk}(q_k^g-\hat{q}_k^g)}{(1-\Delta)^2|\hat{V}_k|^2}+\frac{\Delta R_{jk}(q_k^g-\hat{q}_k^g)}{(1-\Delta)^2|\hat{V}_k|^2}  + \frac{2\Delta R_{jk}\hat{q}_k^g}{(1-\Delta)^2|\hat{V}_k|^2},
\end{multline}
\end{subequations} From the construction of the matrix $\mathbf{A}$, combined with bounds \eqref{eq:first_term_of_both_terms}, \eqref{eq:beta2_1}, \eqref{eq:bounds_alpha_1} and \eqref{eq:bounds_alpha_2}, we see that we can bound the sum of every expression in \eqref{eq:absolute_value_sum} and $\beta_1+\beta_2$ by $((\mathbf{A}+\Delta(\mathbf{B}+\mathbf{C}))(\tilde{\mathbf{s}}-\tilde{\hat{\mathbf{{s}}}}))_{ij}$, for every $i\in\{1,2,3,4\}$ and $j\in \{1,\ldots,4N\}$.
By assumption $\tilde{\mathbf{s}}\in P$, so
\begin{multline*}
((\mathbf{A}+\Delta\mathbf{B})(\tilde{\mathbf{s}}-\tilde{\hat{\mathbf{{s}}}}))+ (\Delta(2\mathbf{B}+(1-\Delta)\mathbf{C})\tilde{\hat{\mathbf{s}}}) \leq \\ \leq \Delta(1-\Delta)\hat{V}_{\text{min}}^3\mathbf{1}^{4N},
\end{multline*} and rewriting gives
$(\mathbf{A}+\Delta\mathbf{B})\tilde{\mathbf{s}} 
 \leq \Delta(1-\Delta)^2\hat{V}_{\text{min}}^3\mathbf{1}^{4N} + (\mathbf{A}-\Delta(\mathbf{B}+(1-\Delta)\mathbf{C})\tilde{\hat{\mathbf{s}}}$.
\end{proof}

\begin{remark}\label{remark:absolute_values}
In the proof of Lemma \ref{lemma:self-mapping}, we handled absolute values of decision variables $p_j$ and $q_j$ as follows:
\begin{align*}
|p_j| = p_j^c + p_j^g, \quad |q_j| = q_j^c + q_j^g.
\end{align*} This is a natural choice, since we have introduced the notation:
\begin{align*}
&p_j = p_j^c - p_j^g,\quad q_j = q_j^c - q_j^g, \\
&p_j^c,p_j^g,q_j^c,q_j^g\geq 0,
\end{align*}
earlier. However, this way of handling absolute values is only correct if, for each $j$, at least one of the values $p_j^c$ and $p_j^g$, and at least one of the values $q_j^c$ and $q_j^g$ is zero. In other words, for each $j$, we need: $p_j^c p_j^g = 0$ and $q_j^c q_j^g = 0$.

Consider the set $P$ (as defined in Theorem \ref{thm:polyhedral_restriction}) as the constraint set in a certain optimization problem. If the optimization problem yields a solution where at least one of the values $p_j^c$ and $p_j^g$, and at least one of the values $q_j^c$ and $q_j^g$ is zero, we handled the absolute values correctly. Thus, suppose that this is not the case, i.e.; suppose there exists at least one combination of variables $p_j^c$ and  $p_j^g$, or $q_j^c$ and $q_j^g$ which are both strictly positive. In what follows, we show that there exists another solution ${p_j'}^c,{p_j'}^g,{q_j'}^c,{q_j'}^g$ with the property that for every $j$, ${p_j'}^c {p_j'}^g = 0$ and ${q_j'}^c {q_j'}^g = 0$, and satisfies the constraints and has the same value for the objective function.

Therefore, let
\begin{align*}
    {p_j'}^c,\ {p_j'}^g = \begin{cases}
        p_j^c,\ p_j^g & \text{if } p_j^c p_j^g = 0,\\
        p_j^c-p_j^g,\ 0 & \text{if } p_j^c p_j^g \neq 0,\ p_j^c \geq p_j^g,\\
        0,\ p_j^g-p_j^c & \text{if } p_j^c p_j^g \neq 0,\ p_j^c < p_j^g,
        \end{cases}\
\end{align*} and in a similar way, let,
\begin{align*}
    {q_j'}^c,\ {q_j'}^g = \begin{cases}
        q_j^c,\ q_j^g & \text{if } q_j^c q_j^g = 0,\\
        q_j^c-q_j^g,\ 0 & \text{if } q_j^c q_j^g \neq 0,\ q_j^c \geq q_j^g,\\
        0,\ q_j^g-q_j^c & \text{if } q_j^c q_j^g \neq 0,\ q_j^c < q_j^g.
        \end{cases}\
\end{align*} and denote ${\mathbf{p}'}^c = ({p_1'}^c,\ldots,{p_N'}^c)^T$, ${\mathbf{p}'}^g = ({p_1'}^g,\ldots,{p_N'}^g)^T$, ${\mathbf{q}'}^c = ({q_1'}^c,\ldots,{q_N'}^c)^T$, ${\mathbf{q}'}^g = ({q_1'}^g,\ldots,{q_N'}^g)^T$ and $\tilde{\mathbf{s}}' = ({\mathbf{p}'}^c,{\mathbf{q}'}^c,{\mathbf{p}'}^g, {\mathbf{q}'}^g)^T$. Thus, the solution ${p_j'}^c,{p_j'}^g,{q_j'}^c,{q_j'}^g$ has the property that for every $j$, ${p_j'}^c {p_j'}^g = 0$ and ${q_j'}^c {q_j'}^g = 0$ by construction.

To see that the solution ${p_j'}^c,{p_j'}^g,{q_j'}^c,{q_j'}^g$ satisfies the constraints,  we show that
\begin{align*}
(\mathbf{A}+\Delta \mathbf{B})(\tilde{\mathbf{s}}'-\tilde{\mathbf{s}}) \leq 0.
\end{align*} This is sufficient, since $\tilde{\mathbf{s}}\in P$.
Notice that, for each $j$, ${p_j'}^c - p_j^c = {p_j'}^g-p_j^g \leq 0$ and ${q_j'}^c-q_j^c = {q_j'}^g-q_j^g \leq 0$, which implies that,
\begin{align*}
\tilde{\mathbf{s}}'-\tilde{\mathbf{s}} = \begin{pmatrix}
{\mathbf{p}'}^c - {\mathbf{p}}^c \\
{\mathbf{q}'}^c - {\mathbf{q}}^c \\
{\mathbf{p}'}^g - {\mathbf{p}}^g \\
{\mathbf{q}'}^g - {\mathbf{q}}^g
\end{pmatrix} =
\begin{pmatrix}
{\mathbf{p}'}^c - {\mathbf{p}}^c \\
{\mathbf{q}'}^c - {\mathbf{q}}^c \\
{\mathbf{p}'}^c - {\mathbf{p}}^c \\
{\mathbf{q}'}^c - {\mathbf{q}}^c
\end{pmatrix} \leq \mathbf{0}.
\end{align*} In other words, the vector $\tilde{\mathbf{s}}'-\tilde{\mathbf{s}}$ has rows $1,\ldots,N$ equal to rows $2N+1,\ldots,3N$ and rows $N+1,\ldots,2N$ equal to rows $3N+1,\ldots,4N$, and all of its entries are values smaller than zero. However, since, for each $j$, ${p_j'}^c - p_j^c = {p_j'}^g-p_j^g$, and ${q_j'}^c-q_j^c = {q_j'}^g-q_j^g$ we also have, for each $j$, ${p_j'}^c - {p_j'}^g = p_j^c-p_j^g$, and ${q_j'}^c-{q_j'}^g  = q_j^c-q_j^g$ which assures that the objective function has the same value. Furthermore, notice that the matrix $A$ also has a particular structure, see \eqref{eq:A}, namely,
$A$ has columns which have opposite signs. The matrix $A$ has columns $1,\ldots,N$ which are identical to columns $2N+1,\ldots,3N$, but with opposite signs, and has columns $N+1,\ldots,2N$ which are identical to columns $3N+1,\ldots,4N$, but with opposite signs. Hence,
\begin{align*}
(\mathbf{A}+\Delta \mathbf{B})(\tilde{\mathbf{s}}'-\tilde{\mathbf{s}}) &= \mathbf{A}(\tilde{\mathbf{s}}'-\tilde{\mathbf{s}}) + \Delta\mathbf{B}(\tilde{\mathbf{s}}'-\tilde{\mathbf{s}}) \\
& = \mathbf{0} + \Delta\mathbf{B}(\tilde{\mathbf{s}}'-\tilde{\mathbf{s}}) \leq \mathbf{0},
\end{align*} since $\Delta\geq 0$ and the matrix $\mathbf{B}$ only contains non-negative elements.
\end{remark}

\subsection{Proof of Lemma \ref{lemma:ineq_voltage_magnitude}}
\begin{proof}
Let $\mathbf{V}\in D$. We show $(1-\Delta)|\hat{V_j}|\leq |V_j|\leq (1+\Delta)|\hat{V_j}|$. First, by the triangle inequality we have
\begin{align*}
    |V_j| & = |V_j-\hat{V_j}+\hat{V_j}|
    \leq |V_j-\hat{V_j}|+|\hat{V_j}|.
\end{align*} Since $\mathbf{V}\in D$, we have for all $j\in\mathcal{N}\backslash\{0\}$, $|V_j-\hat{V_j}|\leq \Delta|\hat{V_j}|$. Hence,
\begin{align}
    |V_j| & \leq \Delta|\hat{V_j}|+|\hat{V_j}|
    = (1+\Delta)|\hat{V_j}|. \label{eq:voltage_drop_1}
\end{align} Second, by the reverse triangle inequality we have
\begin{align*}
    |V_j| = |-V_j| & = |(\hat{V_j}-V_j)-\hat{V_j}|
    \geq |\hat{V_j}|-|\hat{V_j}-V_j|.
\end{align*} Again, since $\mathbf{V}\in D$, we have for all $j\in\mathcal{N}\backslash\{0\}$, $|V_j-\hat{V_j}|\leq \Delta|\hat{V_j}|$ or equivalently, $-|V_j-\hat{V_j}|\geq -\Delta|\hat{V_j}$, which yields
\begin{align}
    |V_j| & \geq |\hat{V_j}|-\Delta \hat{V_j}
    = (1-\Delta)|\hat{V_j}|.\label{eq:voltage_drop_2}
\end{align} Combining \eqref{eq:voltage_drop_1} and \eqref{eq:voltage_drop_2} gives the desired result.
\end{proof}

\subsection{Proof of Lemma \ref{lemma:ineq_voltage_angles}}
\begin{proof} Let $\mathbf{V}\in D$. First, we discuss the inequalities in \eqref{eq:cos_bound}. By the law of cosines, we have
\begin{align*}
|V_j-\hat{V}_j|^2 = |V_j|^2+|\hat{V}_j|^2-2|V_j||\hat{V}_j|\cos(\theta_j-\hat{\theta}_j).
\end{align*} Therefore, we can equivalently write the constraint $|V_j-\hat{V}_j|^2 \leq \Delta^2|\hat{V}_j|^2$ as
\begin{align}
    |V_j|^2+|\hat{V}_j|^2-2|V_j||\hat{V}_j|\cos(\theta_j-\hat{\theta}_j)\leq \Delta^2|\hat{V}_j|^2.\label{eq:equivalent_V}
\end{align} By definition of the cosine and Lemma \ref{lemma:ineq_voltage_magnitude}, we get the inequalities
\begin{align}
|V_j|\cos(\theta_j-\hat{\theta}_j) \leq |V_j| \leq (1+\Delta)|\hat{V}_j|. \label{eq:basic_ineq_V}
\end{align} Combining \eqref{eq:equivalent_V} and \eqref{eq:basic_ineq_V} then yields
\begin{align}
\frac{|V_j|^2}{2|\hat{V}_j|}+\frac{(1-\Delta^2)|\hat{V}_j|}{2} & =  \frac{|V_j|^2+|\hat{V}_j|^2-\Delta^2|\hat{V}_j|^2}{2|\hat{V}_j|} \nonumber \\ & \leq |V_j|\cos(\theta_j-\hat{\theta}_j) \leq  (1+\Delta)|\hat{V}_j|.
\label{eq:bound_cos_guarantee}
\end{align} By Lemma \ref{lemma:ineq_voltage_magnitude}, we have that the first term on the left-hand side of \eqref{eq:bound_cos_guarantee} is bounded by
\begin{align*}
    (1-\Delta)|\hat{V}_j| \leq \frac{|V_j|^2}{2|\hat{V}_j|}+\frac{(1-\Delta^2)|\hat{V}_j|}{2}. 
\end{align*} Hence, we have
\begin{align*}
(1-\Delta)|\hat{V}_j| \leq |V_j|\cos(\theta_j-\hat{\theta}_j) \leq (1+\Delta)|\hat{V}_j|.
\end{align*} Second, we discuss the inequalities in \eqref{eq:sin_bound}. The scalar projection of $V_j$ on $\hat{V}_j$ is given by $v_1:=|V_j|\cos(\theta_j-\hat{\theta}_j)$. Then, define the scalar rejection of the vector $V_j$ on the vector $\hat{V}_j$ as $v_2:= V_j-v_1 = |V_j|\sin(\theta_j-\hat{\theta}_j)$. From the Pythagorean theorem, we know
\begin{align*}
|V_j-\hat{V}_j|^2 = |v_2|^2+|v_1-\hat{V}_j|^2.
\end{align*} As $|V_j-\hat{V}_j|^2$ is the sum of two positive quantities, it can never be smaller than either $|v_2|^2$ or $|v_1-\hat{V}_j|^2$, so the same holds true for $|v_2|, |v_1-\hat{V}_j|$ and $|V_j-\hat{V}_j|$. Since $\mathbf{V}\in D$, it follows
\begin{align*}
||V_j|\sin(\theta_j-\hat{\theta}_j)| = |v_2| \leq |V_j-\hat{V}_j| \leq \Delta|\hat{V}_j|.
\end{align*}
\end{proof}

\subsection{Proof of Lemma \ref{lemma:continuous}}
\begin{proof}
Let $\tilde{\mathbf{s}}\in P$ and $\mathbf{V}\in D$. By assumption, $\Delta<1$, so we have $|V_j| > 0$. The function $x\mapsto 1/x^*$ is continuous on $\mathbb{C}\backslash\{0\}$, and therefore the components of $G$ are continuous on $D$. Thus, $G$ is a continuous operator on $D$.
\end{proof}

\bibliographystyle{plain}
\bibliography{library}

\begin{thebibliography}{10}

\bibitem{Bolognani2016a}
Saverio Bolognani and Sandro Zampieri.
\newblock {On the existence and linear approximation of the power flow solution in power distribution networks}.
\newblock {\em IEEE Transactions on Power Systems}, 31(1):163--172, 2016.

\bibitem{Bose2011a}
Subhonmesh Bose, Dennice~F. Gayme, Steven Low, and K.~Mani Chandy.
\newblock {Optimal power flow over tree networks}.
\newblock {\em 2011 49th Annual Allerton Conference on Communication, Control, and Computing, Allerton 2011}, pages 1342--1348, 2011.

\bibitem{Carvalho2015b}
Rui Carvalho, Lubos Buzna, Richard Gibbens, and Frank Kelly.
\newblock {Critical behaviour in charging of electric vehicles}.
\newblock {\em New Journal of Physics}, 17(9):95001, 2015.

\bibitem{Dvijotham2017a}
Krishnamurthy Dvijotham, Enrique Mallada, and John~W. Simpson-Porco.
\newblock {High-voltage solution in radial power networks: Existence, properties, and equivalent algorithms}.
\newblock {\em IEEE Control Systems Letters}, 1(2):322--327, 2017.

\bibitem{Dvijotham2015a}
Krishnamurthy Dvijotham and Konstantin Turitsyn.
\newblock {Construction of power flow feasibility sets}.
\newblock pages 1--8, 2015.

\bibitem{Gan2015b}
Lingwen Gan, Na~Li, Ufuk Topcu, and Steven~H. Low.
\newblock {Exact Convex Relaxation of Optimal Power Flow in Radial Networks}.
\newblock {\em IEEE Transactions on Automatic Control}, 60(1):72--87, 2015.

\bibitem{Gan2015a}
Lingwen Gan, Na~Li, Ufuk Topcu, and Steven~H. Low.
\newblock {Exact Convex Relaxation of Optimal Power Flow in Radial Networks}.
\newblock {\em IEEE Transactions on Automatic Control}, 60(1):72--87, 2015.

\bibitem{Giraldo2022}
Juan~S. Giraldo, Oscar~Danilo Montoya, Pedro~P. Vergara, and Federico Milano.
\newblock {A fixed-point current injection power flow for electric distribution systems using Laurent series}.
\newblock {\em Electric Power Systems Research}, 211(September 2021):108326, 2022.

\bibitem{Henrion2002}
Didier Henrion and Jean~Bernard Lasserre.
\newblock {GloptiPoly: Global optimization over polynomials with matlab and SeDuMi}.
\newblock {\em Proceedings of the IEEE Conference on Decision and Control}, 1:747--752, 2002.

\bibitem{Kocuk2016}
Burak Kocuk, Santanu~S. Dey, and Xu~Andy Sun.
\newblock {Inexactness of SDP relaxation and valid inequalities for optimal power flow}.
\newblock {\em IEEE Transactions on Power Systems}, 31(1):642--651, 2016.

\bibitem{Lee2019}
Dongchan Lee, Hung~D. Nguyen, Krishnamurthy Dvijotham, and Konstantin Turitsyn.
\newblock {Convex restriction of power flow feasibility sets}.
\newblock {\em IEEE Transactions on Control of Network Systems}, 6(3):1235--1245, 2019.

\bibitem{Lee2020}
Dongchan Lee, Konstantin Turitsyn, Daniel~K. Molzahn, and Line~A. Roald.
\newblock {Feasible Path Identification in Optimal Power Flow with Sequential Convex Restriction}.
\newblock {\em IEEE Transactions on Power Systems}, 35(5):3648--3659, 2020.

\bibitem{Lee2021a}
Dongchan Lee, Konstantin Turitsyn, Daniel~K. Molzahn, and Line~A. Roald.
\newblock {Robust AC Optimal Power Flow with Robust Convex Restriction}.
\newblock {\em IEEE Transactions on Power Systems}, 36(6):4953--4966, 2021.

\bibitem{Liu2008}
Baozhu Liu and Bolong Li.
\newblock {A general algorithm for building Z-matrix based on transitional matrices}.
\newblock {\em 3rd International Conference on Deregulation and Restructuring and Power Technologies, DRPT 2008}, (April):794--797, 2008.

\bibitem{Low2014}
Steven~H. Low.
\newblock {Convex relaxation of optimal power flow - Part i: Formulations and equivalence}.
\newblock {\em IEEE Transactions on Control of Network Systems}, 1(1):15--27, 2014.

\bibitem{Low2014c}
Steven~H. Low.
\newblock {Convex relaxation of optimal power flow-part II: Exactness}.
\newblock {\em IEEE Transactions on Control of Network Systems}, 1(2):177--189, 2014.

\bibitem{Peterson1989}
William~L. Peterson, Elham~B. Makram, and Thomas~L. Baldwin.
\newblock {Generalized PC based bus impedance matrix building algorithm}.
\newblock {\em Conference Proceedings - IEEE SOUTHEASTCON}, 2:432--436, 1989.

\bibitem{Stott2012}
B~Stott and O~Alsa{\c{c}}.
\newblock {Stott-Alsac-OPF-White-Paper}.
\newblock {\em SEPOPE XII Symposium, Rio de Janeiro, Brazil}, pages 1866--1876, 2012.

\bibitem{Wang2018}
Cong Wang, Andrey Bernstein, Jean~Yves {Le Boudec}, and Mario Paolone.
\newblock {Explicit conditions on existence and uniqueness of load-flow solutions in distribution networks}.
\newblock {\em IEEE Transactions on Smart Grid}, 9(2):953--962, 2018.

\bibitem{Wei2017}
Wei Wei, Jianhui Wang, Na~Li, and Shengwei Mei.
\newblock {Optimal Power Flow of Radial Networks and Its Variations: A Sequential Convex Optimization Approach}.
\newblock {\em IEEE Transactions on Smart Grid}, 8(6):2974--2987, 2017.

\bibitem{Wu1982}
Felix~F. Wu and Sadatoshi Kumagai.
\newblock {Steady-State Security Regions of Power Systems}.
\newblock {\em IEEE Transactions on Circuits and Systems}, 29(11):703--711, 1982.

\bibitem{Yu2015a}
Suhyoun Yu, Hung~D. Nguyen, and Konstantin~S. Turitsyn.
\newblock {Simple certificate of solvability of power flow equations for distribution systems}.
\newblock {\em IEEE Power and Energy Society General Meeting}, 2015-Septe:9--13, 2015.

\end{thebibliography}

\begin{IEEEbiography}[{\includegraphics[width=1in,height=1.25in,clip,keepaspectratio]{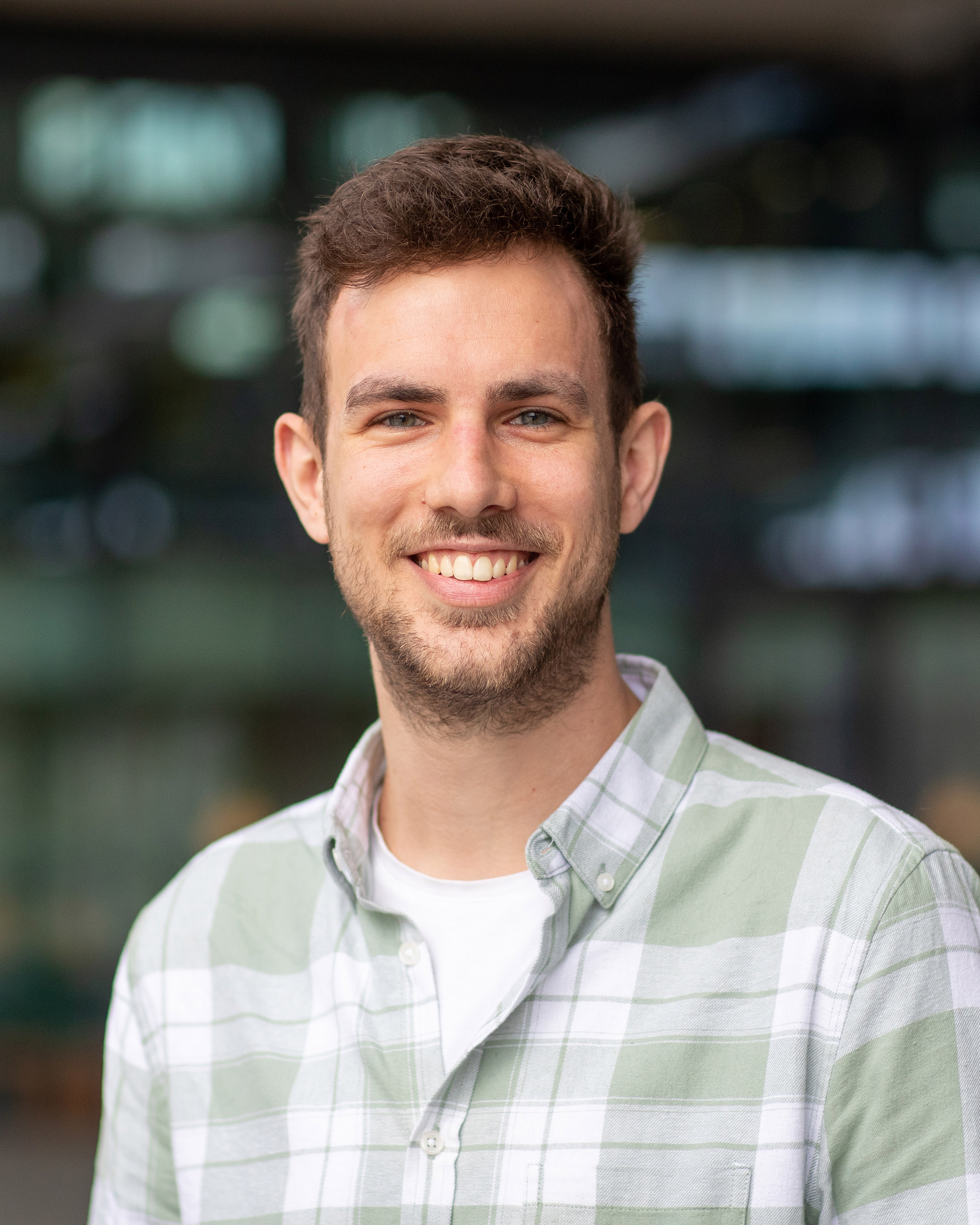}}]{Mark Christianen} received his B.Sc. and M.Sc. degree in Applied Mathematics from Delft University of Technology, Delft, the Netherlands in 2017 and 2019, respectively. 

He is currently working towards his Ph.D. degree at the Department of Mathematics and Computer Science, Eindhoven University of Technology, the Netherlands. 

His research interests include stochastic networks and probability theory.


\end{IEEEbiography}

\begin{IEEEbiography}
[{\includegraphics[width=1in,height=1.25in,clip,keepaspectratio]{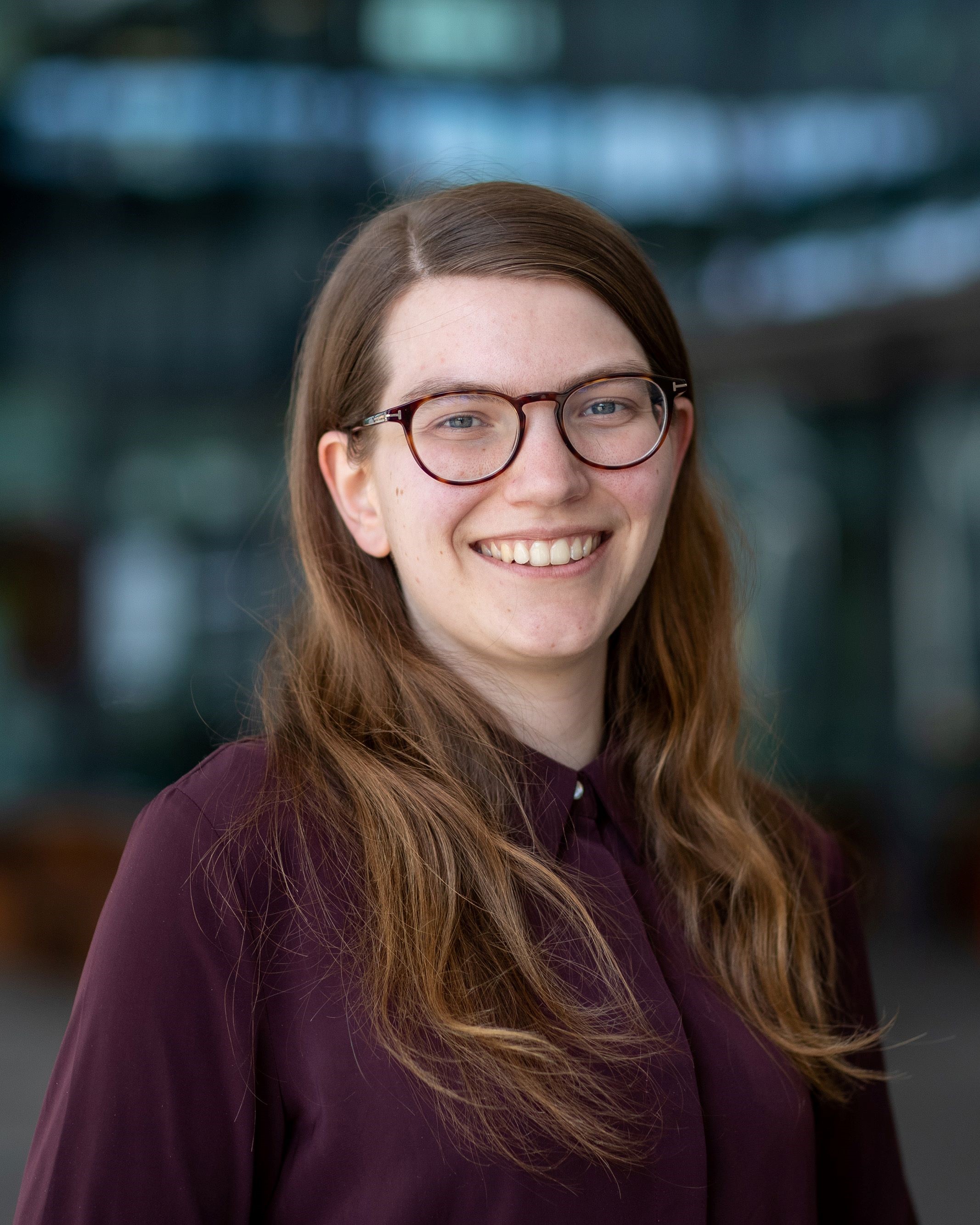}}] {Sanne van Kempen} received her B.Sc. and M.Sc. degree in Applied Mathematics from Eindhoven University of Technology, Eindhoven, the Netherlands in 2020 and 2022, respectively.
 
She is currently working towards her Ph.D. degree at the Department of Mathematics and Computer Science, Eindhoven University of Technology, the Netherlands.
 
Her research interests include queueing networks and adaptive learning techniques.
\end{IEEEbiography}

\begin{IEEEbiography}[{\includegraphics[width=1in,height=1.25in,clip,keepaspectratio]{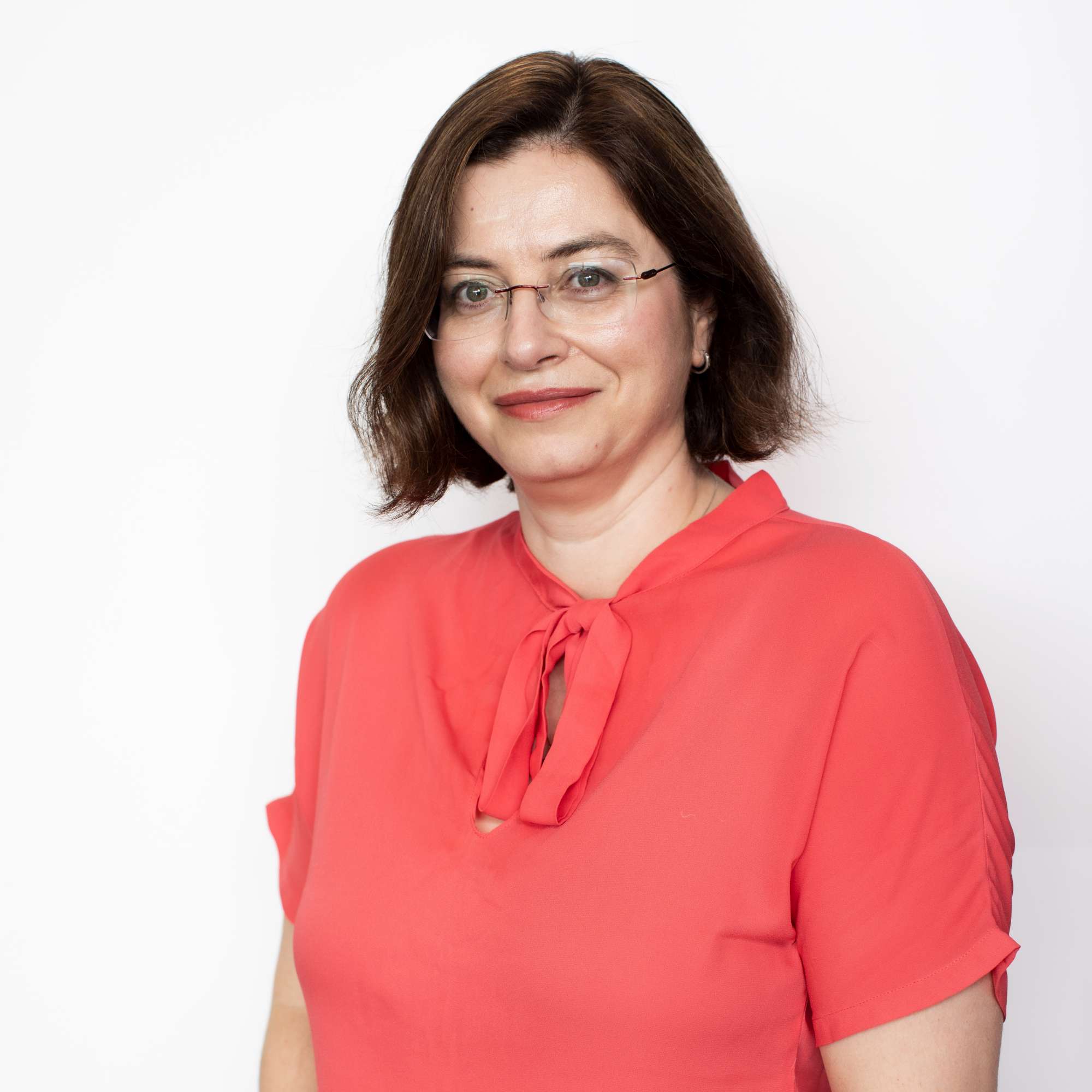}}]{Maria Vlasiou} received her B.Sc. degree (Hons.) from Aristotle University of Thessaloniki in 2002 and completed her Ph.D. in mathematics at TU Eindhoven in 2006. 

She is a Professor of Stochastic Processing Interacting Networks (University of Twente, Eindhoven University of Technology). Her research focuses on the performance of stochastic processes, with applications on manufacturing, telecommunications, mathematical biology, and energy systems. Prof. Vlasiou serves as associate editor in four journals and is the co-chair of the TPCs of IFIP Performance 2023 and NMC 2023. She has received the best paper award in ICORES 2013, the Marcel Neuts student paper award in MAM8, a prize at the 8th conference in Actuarial Science, and the INFORMS UPS G. Smith 2022 award.
\end{IEEEbiography}

\begin{IEEEbiography}[{\includegraphics[width=1in,height=1.25in,clip,keepaspectratio]{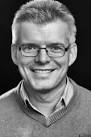}}]{Bert Zwart} (M’18) received his MA degree in econometrics from VU Amsterdam in 1997, and completed his Ph.D. degree in mathematics at TU Eindhoven in 2001. 

He is the Leader of the CWI Stochastics Group (Amsterdam) and Professor with the Eindhoven University of Technology, the Netherlands. His research interests include applied probability and stochastic networks. His work on power systems is focusing on the applications of probabilistic methods to scheduling and reliability issues and rare events such as cascading failures and blackouts.

He has been Stochastic Models area editor for Operations Research, the flagship journal of his profession, from 2009 to 2017 and currently serves on the editorial boards of five journals. He was co-organizer of a special semester on the Mathematics of energy systems taking place in Cambridge, U.K., Spring 2019.

\end{IEEEbiography}

\end{document}